\documentclass[12pt]{amsproc}
\usepackage{amsmath, amssymb, amsthm, mathrsfs}

\usepackage[
a4paper,
left=27mm,
right=27mm,
top=27mm,
bottom=30mm
]{geometry}

\usepackage[colorlinks, backref]{hyperref}
\usepackage{amsrefs}
\usepackage{graphicx}
\usepackage{microtype}
\usepackage{tikz,tikz-cd}

\usepackage[T1]{fontenc}
\usepackage{lmodern}

\theoremstyle{plain}
\newtheorem{theorem}{Theorem}[section]

\newtheorem{lemma}[theorem]{Lemma}
\newtheorem{corollary}[theorem]{Corollary}
\newtheorem{proposition}[theorem]{Proposition}

\theoremstyle{definition}
\newtheorem{definition}[theorem]{Definition}
\newtheorem{notation}[theorem]{Notation}

\theoremstyle{remark}
\newtheorem{remark}[theorem]{Remark}

\newcommand{\cP}{{\mathcal P}}
\newcommand{\cG}{\mathcal{G}}
\newcommand{\bbN}{{\mathbb N}}
\newcommand{\bbR}{{\mathbb R}}
\newcommand{\bbC}{{\mathbb C}}
\newcommand{\bbZ}{{\mathbb Z}}
\newcommand{\pow}{\mathscr{P}}
\DeclareMathOperator{\Ad}{Ad}
\DeclareMathOperator{\ad}{ad}
\DeclareMathOperator{\Aut}{\mathrm{Aut}}
\DeclareMathOperator{\Fin}{\mathrm{Fin}}
\newcommand{\cU}{\mathcal U}

\DeclareMathOperator{\antr}{antr}
\DeclareMathOperator{\ntr}{ntr}
\DeclareMathOperator{\tr}{tr}
\DeclareMathOperator{\Alt}{Alt}

\DeclareMathOperator{\U}{\mathrm{U}}
\DeclareMathOperator{\M}{\mathrm{M}}
\DeclareMathOperator{\MHS}{\mathrm{M}^{\mathrm{HS}}_\infty}
\DeclareMathOperator{\UHS}{\mathrm{U}^{\mathrm{HS}}}

\newcommand{\dmet}{d}
\newcommand{\dHS}{d^{\mathrm{HS}}} 
\newcommand{\dHSk}[1]{d^{\mathrm{HS}}_{#1}} 
\newcommand{\dHSn}{d^{\mathrm{HS}}_n} 
\newcommand{\dHSinf}{d^{\mathrm{HS}}_\infty} 
\DeclareMathOperator{\PU}{\mathrm{PU}}
\DeclareMathOperator{\SU}{\mathrm{SU}}
\DeclareMathOperator{\vol}{\mathrm{vol}}
\DeclareMathOperator{\dist}{\mathrm{dist}}
\DeclareMathOperator{\diam}{\mathrm{diam}}
\newcommand{\seqtoinf}{(\bbN^\bbN)_\infty}

\title[Rigidity in trace norm]{Trace-norm rigidity for reduced products of \\ unitary groups and matrix algebras}

\author{Ben De Bondt}
\address{Ben De Bondt, Universität Münster, 48149 Münster, Germany}
\email{bdebondt@uni-muenster.de}

\author{Andreas Thom}
\address{Andreas Thom, TU Dresden, 01062 Dresden, Germany}
\email{andreas.thom@tu-dresden.de}

\begin{document}
	
	\begin{abstract}
		We study homomorphisms, with a focus on isomorphisms, between the tracial metric reduced products of finite dimensional unitary groups and of matrix algebras. A variant of Ulam stability for unitary groups and a classification of the almost surjective continuous homomorphisms between finite dimensional unitary groups are proved and then used to show that all isomorphisms of product form of these tracial reduced products are induced by almost permutations of the coordinates and coordinatewise application of automorphisms. 
		We prove coordinate recognition for these reduced products and obtain under set theoretic assumptions rigidity and classification results for their full automorphism groups. For tracial reduced matrix algebras we obtain such rigidity result in the more general context of center-preserving $*$-homomorphisms.
	\end{abstract}
	
	\maketitle
	\tableofcontents
	
	\section{Introduction}
	Isomorphisms between metric reduced products of finite symmetric groups were studied in~\cite{debondtthom}, where it was shown that such isomorphisms can consistently all be obtained, up to canonical identification of asymptotically equivalent dimensions, simply from composition of inner automorphisms and suitable almost permutations of $\bbN$, whereas in contrast, under the continuum hypothesis $\mathsf{CH}$, these isomorphisms exhibit much wilder behaviour.
	As remarked in~\cite{debondtthom}, it is very natural to ask whether the results obtained there for universal sofic groups of reduced product type translate to universal hyperlinear groups. We thus pursue here the natural continuation of this work, studying metric reduced products with respect to metrics induced by the normalized \emph{Hilbert--Schmidt norm}, or \emph{trace norm}:
	\[
	\|x\|_{2,n}=\tau_n(x^*x)^{1/2},
	\qquad \tau_n=\frac{1}{n}\operatorname{tr},
	\]
	for $x\in \M_n(\mathbb C)$. Closely related to the questions we address here, there is inspiring previous work in the operator norm setting, see \cite{Farah-AC, Farah-C, MR3971570, McKenney-Vignati} and the more recent \cite{AlekFaraThom}, yet on a technical level, the tracial phenomena we study require different and new ideas.
	
	For a sequence $(k_n)_n$ with $k_n\to\infty$, we
	write
	$
	\UHS[(k_n)_n]
	$
	for the reduced product of the groups $\U(k_n)$ (the group of unitary complex $k_n \times k_n$-matrices) with respect to the
	normalized Hilbert--Schmidt metric.  We first study isomorphisms of product form between such groups, that is, isomorphisms implemented coordinatewise by maps
	$h_n:\U(k_n)\to \U(l_n)$.  Proposition~\ref{prop_stab} shows that if
	\[
	\UHS[(k_n)_n]\cong \UHS[(l_n)_n]
	\]
	by such a product-form isomorphism, then $k_n/l_n\to 1$.  Moreover, after
	identifying asymptotically equivalent dimensions, every product-form
	automorphism is induced coordinatewise by automorphisms of finite unitary
	groups.  The proof combines compact-group stability (Corollary~\ref{app:cor-dCOT-sequential}) with a classification
	result (Theorem~\ref{th_stab}) for almost surjective homomorphisms $\U(n)\to \U(m)$, that is of independent interest and is the subject of Section~\ref{sec:almost surj cont hom} of the paper.
	
	We then prove a coordinate-recognition theorem for arbitrary group
	isomorphisms between such reduced products.  Theorem~\ref{th.action boundary}
	shows that every isomorphism
	\[
	\UHS[(k_n)_n]\longrightarrow \UHS[(l_n)_n]
	\]
	induces an automorphism of the Stone--\v{C}ech boundary $\partial\beta\mathbb N$
	and transports according to this
	automorphism the pseudometrics on $\UHS[(k_n)_n]$ that measure the distance between elements of $\UHS[(k_n)_n]$ after restricting them to some infinite set of coordinates.  This part of the argument
	uses two group-theoretic features of finite unitary groups: uniform bounded
	normal generation modulo the center (\cite{dowerkthom}*{Theorem~1.1 and Lemma~7.8}, see Lemma~\ref{th.boundedgeneration}), and the definability of the absolute
	normalized trace on involutions (Lemma~\ref{flower}).  These allow one to recover the Boolean
	algebra of coordinate supports from the abstract group structure.
	
	We next put all this together to prove the unitary-group analogue of the structural analysis carried out for symmetric groups by the authors in~\cite{debondtthom}. In both the symmetric group and unitary group setting, the full rigidity and automorphism-group results require the use of additional axioms beyond $\mathsf{ZFC}$ and are obtained under Todor\v{c}evi\'c's Open Coloring Axiom $\mathrm{OCA}$ together with the fragment of Martin's axiom $\mathrm{MA}_{\aleph_1}(\sigma\text{-linked})$.
	
	The proof patterns for deriving such rigidity results from these axioms are very similar for both families of groups, and it is instructive to formulate this template here in a general context.
	
	\begin{notation}\label{not.not1}$\;$\\
		Let $\cG = \{ (G_k,\dmet_k) : k \in \bbN \}$ be a sequence of separable bi-invariant metric groups, uniformly bounded in diameter. (We could as well formulate these considerations for more general families of metric structures in the sense of metric model theory, we restrict to the group setting for concreteness without losing generality in the arguments.) \\
		Let $\seqtoinf$ be the set of sequences $(k_n)_n$ of natural numbers satisfying $k_n \to \infty$ as $n\to \infty$.   We write
		\[
		\mathcal G[(k_n)_n]:=\prod_n G_{k_n}/\Fin      \qquad
		\text{for }(k_n)_n \in \seqtoinf,
		\]
		i.e. $\mathcal G[(k_n)_n]$ is the quotient group of $\prod_n G_{k_n}$ with respect to the normal subgroup given by all sequences $(a_n)_n \in \prod_n G_{k_n}$ with $\dmet_{k_n}(a_n,1) \to 0.$ We denote by $[a_n]_n$ the element of $\prod_n G_{k_n}/\mathrm{Fin}$ represented by $(a_n)_n \in \prod_n G_{k_n}$. For $a= [a_n]_n, b =[b_n]_n \in \mathcal G[(k_n)_n]$ and a set $S \subseteq \mathbb{N}$, write 
		\[ a =_S b \text{ for } \lim_{n\in S} \dmet_{k_n}(a_n,b_n) = 0.   \]
	\end{notation}
	
	Now assume that there exists an equivalence relation $\asymp$ on $\seqtoinf$ such that for all $(k_n)_n, (l_n)_n \in \seqtoinf$ with $(k_n)_n \asymp (l_n)_n$ there exists a canonical group isomorphism of product form
	\[
	\Gamma[ (k_n)_n, (l_n)_n ]: \cG[(k_n)_n] \to \cG[(l_n)_n].
	\]
	Here a map $\cG[(k_n)_n] \to \cG[(l_n)_n]$ is said to be \emph{of product form} if it is given by $[a_n]_n \mapsto [h_n(a_n)]_n$ for some sequence of maps $h_n : G_{k_n} \to G_{l_n}$.
	Naturally occurring examples of such equivalence relations $\asymp$ are asymptotic equality (i.e. $k_n/l_n \rightarrow 1$) and eventual identity (i.e. $k_n = l_n$ for all but finitely many $n$.)
	
	Assume in addition that
	\begin{enumerate}
		\item[(CR)] \textbf{$\cG$ recognises coordinates.}
		For every isomorphism
		$
		\varphi:\mathcal G[(k_n)_n]\to \mathcal G[(l_n)_n]
		$
		there exists an automorphism
		$
		\theta \in \Aut(\pow(\mathbb N)/\Fin)
		$
		such that for all $a,b\in \mathcal G[(k_n)_n]$ and all $S\in \pow(\mathbb N)/\Fin$,
		\[
		a=_S b
		\iff
		\varphi(a)=_{\theta(S)}\varphi(b).
		\]
		
		\item[(PF)] $\cG$ has \textbf{Product-form rigidity.}
		Whenever
		$
		\psi:\mathcal G[(k_n)_n]\to \mathcal G[(l_n)_n]
		$
		is an isomorphism of product form,
		then $(k_n)_n \asymp (l_n)_n$ and there exists a sequence
		$
		\alpha_n\in\Aut(G_{l_n})
		$
		such that
		\[
		\psi
		=
		\varphi[(\alpha_n)_n]\circ \Gamma[(k_n)_n,(l_n)_n],
		\]
		where $\varphi[(\alpha_n)_n]: \mathcal G[(l_n)_n]\to \mathcal G[(l_n)_n] $ is the automorphism of product form mapping $[a_n]_n$ to $[\alpha_n(a_n)]_n$.
	\end{enumerate}
	
	In this context an isomorphism
	$
	\Phi:\mathcal G[(k_n)_n]\to \mathcal G[(l_n)_n]
	$
	is called \emph{trivial} if there exist:
	\begin{itemize}
		\item an almost permutation $f:\mathbb N\to\mathbb N$,
		\item and a sequence of automorphisms
		$
		\alpha_n\in \Aut(G_{l_n}),
		$
	\end{itemize}
	such that
	$
	(k_{f(n)})_n \asymp (l_n)_n,
	$
	and
	\[
	\Phi
	=
	\varphi[(\alpha_n)_n] \circ \Gamma[ (k_{f(n)})_n ,(l_n)_n] \circ \psi_f,
	\]
	where
	$
	\psi_f:\mathcal G[(k_n)_n]\to \mathcal G[(k_{f(n)})_n]
	$
	is the reindexing isomorphism
	$
	\psi_f([a_n]_n)=[a_{f(n)}]_n,
	$
	and $\varphi[(\alpha_n)_n]: \mathcal G[(l_n)_n]\to \mathcal G[(l_n)_n] $ is the automorphism of product form mapping $[a_n]_n$ to $[\alpha_n(a_n)]_n$.
	
	The mechanism behind the proof of the main result from \cite{debondtthom}, which we apply here for unitary groups, can then be isolated in the following Lemma~\ref{lem.rigidGeneral}. We note that this precisely motivates the results aimed at in Section~\ref{sec:compact-domain-stability} and Section~\ref{sec:almost surj cont hom} and in the coordinate recognition argument of Section~\ref{sec:reduced-products}, because these furnish the premises required for application of Lemma~\ref{lem.rigidGeneral}. Furthermore, with exception of the proof of Theorem~\ref{thm:center-preserving-tracial-wep}, all applications of forcing axioms in this paper factor through application of Lemma~\ref{lem.rigidGeneral}. Moreover, Section~\ref{sec:compact-domain-stability}, Section~\ref{sec:almost surj cont hom} and the coordinate recognition principles of Section~\ref{sec:reduced-products} and Section~\ref{sec:matrix-reduced-products} require only $\mathsf{ZFC}$-arguments.
	
	\begin{lemma}\label{lem.rigidGeneral}
		Assume
		$
		\mathrm{OCA}+\mathrm{MA}_{\aleph_1}(\sigma\text{-linked}).
		$
		Let $\mathcal G=\{(G_k,\dmet_k):k\in\mathbb N\}$ be a sequence of separable bi-invariant metric groups of uniformly bounded diameter endowed with an equivalence relation $\asymp$ on $\seqtoinf$ and canonical group isomorphisms
		$
		\Gamma[ (k_n)_n, (l_n)_n ]: \cG[(k_n)_n] \to \cG[(l_n)_n]
		$
		for $(k_n)_n \asymp (l_n)_n$ and assume that $\mathcal G$ satisfies \emph{(CR)} and \emph{(PF)} above.
		Then every isomorphism
		$
		\varphi:\mathcal G[(k_n)_n]\to \mathcal G[(l_n)_n]
		$
		is trivial.
	\end{lemma}
	
	The proof of this lemma is only notationally different from the proof of Theorem 3.9 in \cite{debondtthom}, we include it for completeness.
	
	\begin{proof}
		Let
		$\varphi:\mathcal G[(k_n)_n]\to \mathcal G[(l_n)_n]$
		be an isomorphism.    By hypothesis (CR), there exists $\theta:=\theta_\varphi\in \Aut(\pow(\mathbb N)/\Fin)$
		such that
		\[a=_S b\iff\varphi(a)=_{\theta(S)}\varphi(b)\]
		for all $a,b\in \mathcal G[(k_n)_n]$ and all $S\in \pow(\mathbb N)/\Fin$.
		Under $\mathrm{OCA}$, every automorphism of $\pow(\mathbb N)/\Fin$ is induced by an almost permutation of $\mathbb N$. Hence there exists an almost permutation
		$f:\mathbb N\to\mathbb N$
		such that $\theta(S)=f^{-1}[S]$,
		for every $S\in\pow(\mathbb N)/\Fin$. 
		Define the reindexing isomorphism
		\[\psi_{f^{-1}}:\mathcal G[(k_{f(n)})_n]\to \mathcal G[(k_n)_n],\qquad\psi_{f^{-1}}([x_n]_n)=[x_{f^{-1}(n)}]_n,\]
		and set
		$\widehat\varphi:=\varphi\circ \psi_{f^{-1}}:\mathcal G[(k_{f(n)})_n]\to \mathcal G[(l_n)_n].$
		Note that $\widehat\varphi$ then fixes all coordinate equivalence relations $=_S$, i.e.
		\[a=_S b\iff\widehat\varphi(a)=_S\widehat\varphi(b)\qquad
		\text{for all }a,b\in\mathcal G[(k_{f(n)})_n],\ S\in\pow(\mathbb N)/\Fin.\]
		
		By hypothesis, the groups $G_k$ are separable metric spaces of uniformly bounded diameter, so the metric lifting theorem from~\cite{debondtvignati} applies: under $\mathrm{OCA}+\mathrm{MA}_{\aleph_1}(\sigma\text{-linked})$ every coordinate-fixing isomorphism between reduced products of such metric spaces is of product form.
		
		So there exist maps
		$h_n:G_{k_{f(n)}}\to G_{l_n}$
		such that
		$\widehat\varphi([(x_n)_n])=[(h_n(x_n))]_n$
		for all $[(x_n)_n]\in \mathcal G[(k_{f(n)})_n].$
		Hypothesis (PF) now applies to $\widehat\varphi$. Therefore
		$(k_{f(n)})_n \asymp \mathbf (l_n)_n$
		and there exists a sequence
		$\alpha_n\in \Aut(G_{l_n})$
		such that
		$\widehat\varphi=\varphi[(\alpha_n)_n]\circ \Gamma[(k_{f(n)})_n,(l_n)_n].$
		
		Finally,
		$\varphi=\widehat\varphi\circ \psi_f=\varphi[(\alpha_n)_n]\circ \Gamma[(k_{f(n)})_n,(l_n)_n]\circ \psi_f.$ Hence $\varphi$ is trivial.
	\end{proof}
	
	Because of the close relation to unitary groups, we consider in Section~\ref{sec:matrix-reduced-products} corresponding questions for reduced products of matrix algebras with respect to the normalised trace norm.  In the algebraic case
	the coordinate recognition argument is simpler, because the center is already the commutative
	$C^*$-algebra $C(\partial\beta\mathbb N)$, and hence the Stone--\v{C}ech
	boundary is visible directly from the ring structure.
	
	The paper is organized as follows.  Section~\ref{sec:compact-domain-stability} proves the
	compact-group stability theorem in normalized Hilbert-Schmidt metric.
	Section~\ref{sec:almost surj cont hom} contains the proof of the classification theorem for almost surjective homomorphisms between finite unitary groups, which passes through certain representation-theoretic estimates, and Section~\ref{sec:reduced-products} studies
	product-form isomorphisms and coordinate recognition for reduced products
	of unitary groups.  Section~\ref{sec:matrix-reduced-products} treats the
	parallel statement for reduced products of matrix algebras.  
	
	\section{Compact-group stability in trace norm}\label{sec:compact-domain-stability}
	
	For $n\in \bbN$ and $x\in \M_{n}(\mathbb C)$, write
	\[
	\|x\|_{2,n}:=\left(\frac{1}{n}\tr(x^*x)\right)^{1/2}
	\]
	for the normalized trace norm. We denote by $\dHSk{n}$ the metric on $\U(n)$ induced by $\|\cdot\|_{2,n}$.
	
	In this section we prove a stability theorem for uniform almost homomorphisms from a compact metrizable group into a finite-dimensional unitary group with respect to the normalized trace norm. The result is the compact-group analogue of a theorem for amenable groups from~\cite{dCOT}, and its proof is a direct adaptation of their argument. Compared with the discrete amenable setting, two additional issues must be addressed. Firstly, an approximate representation need not be measurable. This is handled by constructing a finite-valued measurable approximation under a quasi-Lipschitz hypothesis. Secondly, the argument proceeds by considering a regularized positive definite kernel obtained by Haar averaging, which is a priori only measurable and not necessarily continuous. This point is addressed by replacing the measurable positive definite kernel with its unique continuous representative via a version of Pettis' theorem~\cite{Pettis1950}.

	The theorem may also be viewed as an extension of the finite-group Hilbert--Schmidt stability theorem of Gowers and Hatami~\cite{Gowers_Hatami}. Their paper already suggests that such a generalization should hold, although the details are not worked out there. We found it more convenient to follow the methods of \cite{dCOT}, but we do not claim any originality for the result. A more restrictive version of this generalization appears in \cite{AlekseevThom} in the form of a stability theorem for uniform almost homomorphisms from finite unitary groups to the unitary group of a tracial von Neumann algebra, which is proved by a modification of the same argument appearing in this section. For ease of the reader we still include the full argument.
	
	\begin{definition}
		Let $K$ be a group and let $\varphi\colon K\to \U(n)$ be a map.
		\begin{enumerate}
			\item[(i)] We say that $\varphi$ is an $\varepsilon$-representation if
			\[
			\dHSk{n}(\varphi(gh),\varphi(g)\varphi(h))\le \varepsilon
			\qquad (g,h\in K).
			\]
			\item[(ii)] Given a metric $\dmet_K$ on $K$ and a constant $L\ge 0$, we say that $\varphi$ is $(L,\varepsilon)$-quasi-Lipschitz if
			\[
			\dHSk{n}(\varphi(g),\varphi(h))\le L \dmet_K(g,h)+\varepsilon
			\qquad (g,h\in K).
			\]
		\end{enumerate}
	\end{definition}
	
	Our main result will be proved after a sequence of lemmas. The first one is a measurable approximation result for quasi-Lipschitz $\varepsilon$-representations.
	
	\begin{lemma}\label{app:lem-measurable-approx}
		Let $K$ be a compact metrizable group equipped with a compatible bi-invariant metric $\dmet_K$, and let $\varphi\colon K\to \U(n)$ be an $(L,\varepsilon)$-quasi-Lipschitz $\varepsilon$-representation. Then for every $r>0$ there exists a Borel map $\psi\colon K\to \U(n)$ with finite image such that
		\[
		\sup_{g\in K} \dHSk{n} (\varphi(g),\psi(g))\le Lr+\varepsilon
		\]
		and
		$
		\dHSk{n}(\psi(gh),\psi(g)\psi(h))\le 3Lr+4\varepsilon$
		(for all $g,h\in K$).
	\end{lemma}
	
	\begin{proof} Let $r>0$.
		Fix a compatible metric on $K$ and let $\{x_1,\dots,x_N\}\subseteq K$  be a finite $r$-net. Note that then there exists a Borel partition $K=A_1\sqcup\cdots\sqcup A_N$ with $A_i\subseteq B(x_i,r)$. Set
		$
		\psi(g):=\varphi(x_i)$ 
		(for $g\in A_i$).
		Then $\psi$ is Borel and has finite image. If $g\in A_i$, then
		\[
		\dHSk{n}(\varphi(g),\psi(g))=\dHSk{n}(\varphi(g),\varphi(x_i))\le L\dmet_K(g,x_i)+\varepsilon\le Lr+\varepsilon.
		\]
		For the multiplicative defect, use the bi-invariance of $\dHSk{n}$ on $\U(n)$:
		\begin{align*}
			\dHSk{n}(\psi(gh),\psi(g)\psi(h))
			&\le \dHSk{n}(\psi(gh),\varphi(gh))
			+\dHSk{n}(\varphi(gh),\varphi(g)\varphi(h)) \\
			&\qquad +\dHSk{n}(\varphi(g)\varphi(h),\psi(g)\psi(h))\\
			&\le (Lr+\varepsilon)+\varepsilon+\dHSk{n}(\varphi(g),\psi(g))+\dHSk{n}(\varphi(h),\psi(h))\\
			&\le 3Lr+4\varepsilon.
		\end{align*}
	\end{proof}
	
	\begin{definition}
		Let $K$ be a group. A function $\Phi\colon K\to \M_n(\bbC)$ is called positive definite if for every finite set $F\subseteq K$ the block matrix
		\[
		[\Phi(xy^{-1})]_{x,y\in F}\in \M_{|F|}(\M_n(\bbC))
		\]
		is positive semi-definite.
	\end{definition}
	
	When $K$ is compact, measurability on~$K$ is always understood with respect to Haar probability measure. The following proposition is a matrix-valued version of Pettis' theorem~\cite{Pettis1950} for positive definite functions.
	
	\begin{proposition}\label{app:prop-pd-continuous}
		Let $K$ be a compact metrizable group, and let $\Phi\colon K\to \M_n(\bbC)$ be a measurable positive definite function. Then there is a unique continuous positive definite function $\Phi_c\colon K\to \M_n(\bbC)$ such that $\Phi=\Phi_c$ almost everywhere.
	\end{proposition}
	
	\begin{proof}
		For $\xi,\eta\in\bbC^n$, the scalar coefficient
		$
		g\mapsto \langle \Phi(g)\xi,\eta\rangle
		$
		is measurable, and for $\xi=\eta$ it is positive definite as well. By Dixmier's classical theorem for scalar-valued measurable positive definite functions on locally compact groups, each diagonal coefficient has a unique continuous positive definite representative; see \cite{Dixmier}*{\S13.4}. Polarization then provides continuous representatives for all matrix coefficients. These coefficients assemble into a continuous map $\Phi_c\colon K\to \M_n(\bbC)$ satisfying $\Phi=\Phi_c$ almost everywhere. By continuity, positivity of the finite block matrices passes to $\Phi_c$, and uniqueness is immediate because two continuous functions on $K$ that agree almost everywhere must agree everywhere.
	\end{proof}
	
	The main theorem of this section is the following.
	
	\begin{theorem}\label{app:thm-dCOT}
		For every $L>0$ and $\delta>0$, there exists $\varepsilon>0$ such that the following holds. Let $K$ be a compact metrizable group, and let $\varphi\colon K\to \U(n)$ be an $(L,\varepsilon)$-quasi-Lipschitz $\varepsilon$-representation. Then there exist $m \in [n,(1+\delta)n] \cap \bbN$, a continuous group homomorphism $\rho \colon K \to \U(m)$, and an isometry $V \colon \bbC^n \to \bbC^m$ such that
		\[
		\|\varphi(g)-V^*\rho(g)V\|_{2,n}<\delta.
		\]
	\end{theorem}
	
	\begin{proof}
		Fix $L>0$ and $\delta>0$. After inserting two preliminary reductions that
		are specific to compact groups, the argument is the same as in
		\cite{dCOT}*{Theorem~5.2 and the proof of Theorem~1.6}. First apply
		Lemma~\ref{app:lem-measurable-approx} with a small parameter $r>0$. This
		yields a Borel map $\psi\colon K\to \U(n)$ with finite image such that
		\[
		\sup_{g\in K}\dHSk{n}(\varphi(g),\psi(g))\le Lr+\varepsilon
		\]
		and
		$
		\dHSk{n}(\psi(gh),\psi(g)\psi(h))\le 3Lr+4\varepsilon$
		(for $g,h\in K$).
		
		Thus, after choosing $r$ and then $\varepsilon$ sufficiently small in
		terms of $L$ and $\delta$, we may replace $\varphi$ by the measurable
		approximate representation $\psi$, while retaining the same type of
		small-defect estimates as in \cite{dCOT}. Since the approximation is
		uniform, any final estimate for $\psi$ carries over to $\varphi$ by the
		triangle inequality.
		
		Next, as in \cite{dCOT}, one averages $\psi$ over the group to obtain a
		matrix-valued positive definite kernel. In the discrete amenable setting,
		this kernel is available pointwise from the start, whereas here Haar
		averaging produces only a measurable positive definite function.
		Proposition~\ref{app:prop-pd-continuous} then allows us to replace it by
		its unique continuous positive definite representative. Because the two
		functions agree almost everywhere, every inequality established in
		\cite{dCOT} for the averaged kernel remains valid for the continuous
		representative. After this replacement, the rest of the proof of
		\cite{dCOT}*{Theorem~5.2} applies verbatim: one performs the
		GNS/Stinespring construction for the continuous positive definite kernel
		and obtains an integer $m\in[n,(1+\delta)n]\cap\bbN$, a continuous
		homomorphism $\rho\colon K\to \U(m)$, and an isometry
		$V\colon \bbC^n\to \bbC^m$ such that $\psi(g)$ is uniformly close to
		$V^*\rho(g)V$.
		
		Finally, the passage from $\psi$ back to $\varphi$ is exactly the same as
		in the proof of \cite{dCOT}*{Theorem~1.6}. Combining the estimate from
		the previous paragraph with
		\[\sup_g\dHSk{n}(\varphi(g),\psi(g))\le Lr+\varepsilon\] and choosing the
		initial parameters sufficiently small yields
		$
		\|\varphi(g)-V^*\rho(g)V\|_{2,n}<\delta$
		(for $g\in K$).
		This proves the theorem.
	\end{proof}
	
	\begin{corollary}\label{app:cor-dCOT-sequential}
		Let $(K_n,\dmet_n)_n$ be a sequence of compact metrizable groups with compatible bi-invariant metrics. Let $N_n\to\infty$, and let
		$
		\varphi_n\colon K_n\to \U(N_n)
		$
		be maps. Assume that there are numbers $\varepsilon_n\to0$ and $L_n\ge0$ such that $\varphi_n$ is an $\varepsilon_n$-representation and is $(L_n,\varepsilon_n)$-quasi-Lipschitz. 
		Then, after discarding finitely many indices, there exist integers
		\(m_n\ge N_n\), continuous homomorphisms \(\rho_n:K_n\to U(m_n)\), and
		isometries \(V_n:\mathbb C^{N_n}\to\mathbb C^{m_n}\) such that
		$m_n/N_n\to 1$
		and
		\[
		\sup_{g\in K_n}
		\|\varphi_n(g)-V_n^*\rho_n(g)V_n\|_{2,N_n}\to 0.
		\]
	\end{corollary}
	
	\begin{proof}
		Choose $r_n>0$ such that $r_n\to0$ and $L_nr_n\to0$. By Lemma~\ref{app:lem-measurable-approx}, there are Borel maps
		$
		\psi_n\colon K_n\to \U(N_n)
		$
		with finite image such that
		\[
		\sup_{g\in K_n}\dHSk{N_n}(\varphi_n(g),\psi_n(g))\to0
		\]
		and
		\[
		\sup_{g,h\in K_n}
		\dHSk{N_n}(\psi_n(gh),\psi_n(g)\psi_n(h))\to0.
		\]
		Applying the proof of Theorem~\ref{app:thm-dCOT} successively with \(\delta=1/k\) and using a
		diagonal argument, we obtain, after discarding finitely many indices,
		integers \(m_n\ge N_n\), continuous homomorphisms
		\(\rho_n:K_n\to U(m_n)\), and isometries
		\(V_n:\mathbb C^{N_n}\to\mathbb C^{m_n}\) such that
		$m_n/N_n\to 1$ and
		\[
		\sup_{g\in K_n}
		\|\psi_n(g)-V_n^*\rho_n(g)V_n\|_{2,N_n}\to0.
		\]
		The conclusion follows from the triangle inequality.
	\end{proof}
	
	\section{Almost surjective continuous homomorphisms}\label{sec:almost surj cont hom}

	The goal of this section is to prove the following structure result on almost surjective unitary representations. The result seems plausible and maybe even obvious to experts, however, we could not find a direct and elementary argument.
	
	\begin{definition}
		A map $f$ from a metric space $(M,d)$ to a metric space $(N,\partial)$ is $\delta$-surjective if for every $y\in N$ there exists $x\in M$ such that $\partial(f(x),y)\le \delta$.
	\end{definition}
	
	\begin{theorem}  \label{th_stab}
		For every $\varepsilon>0$ there exists $\delta>0$ such that the following holds.
		If $\varphi: \U(n) \to \U(m)$ is a continuous homomorphism which is $\delta$-surjective w.r.t.\ the normalized trace norm, then one of the following holds:
		\begin{enumerate}
			\item[(i)] $m=1$, in which case $\varphi(v)=\det(v)^k$ for some $k\in \bbZ$.
			\item[(ii)] $n\le m\le (1+\varepsilon)n$, and there exist an integer $k\in \bbZ$ and $u\in \U(m)$ such that either
			\[
			\forall v\in \U(n)\qquad
			\dHSk{m}\bigl(\varphi(v),\det(v)^k\,\ad_u(v\oplus I_{m-n})\bigr)\le \varepsilon,
			\]
			or
			\[
			\forall v\in \U(n)\qquad
			\dHSk{m}\bigl(\varphi(v),\det(v)^k\,\ad_u(\overline{v}\oplus I_{m-n})\bigr)\le \varepsilon.
			\]
		\end{enumerate}
		In particular, in case \textup{(ii)}, if one imposes in addition the normalization
		$
		\varphi(zI_n)=zI_m
		$
		for all $z\in \mathbb{T}$, then $k=0$, unless $(n,k)=(2,1)$.
	\end{theorem}
	
	For $n\ge 1$, let $\mathbb{T}\subset \U(n)$ denote the scalar subgroup, and set
	$
	\PU(n):=\U(n)/\mathbb{T}.
	$
	We equip $\PU(n)$ with the quotient metric
	$
	\bar \dmet_n(\bar u,\bar v):=
	\inf_{\omega\in\mathbb{T}} \dHSk{n}(\omega u,v).
	$
	The quotient map $q_n:\U(n)\to \PU(n)$ is $1$-Lipschitz. Let $\delta_n$ denote the quotient geodesic distance on $\PU(n)$ induced by the bi-invariant Riemannian metric on $\U(n)$ associated with $\dHSk{n}$.
	
	We shall also use the standard parametrization of irreducible representations of $\U(n)$ by highest weights; see, for example, \cite{FH} and \cite{Hall}. A highest weight is an $n$-tuple
	$
	\lambda=(\lambda_1,\dots,\lambda_n)\in \bbZ^n,$
	with
	$\lambda_1\ge \lambda_2\ge \cdots \ge \lambda_n.
	$
	Let $V_\lambda$ be the corresponding irreducible representation. Its dimension is given by Weyl's formula
	\begin{equation}\label{app:eq:weyl-dim}
		\dim V_\lambda
		=
		\prod_{1\le i<j\le n}
		\frac{\lambda_i-\lambda_j+j-i}{j-i}.
	\end{equation}
	If $k=\lambda_n$ and
	$
	\mu_i:=\lambda_i-\lambda_n$ for $1\le i\le n$,
	then $\mu_n=0$, $\mu_1\ge \cdots \ge \mu_n=0$, and
	$
	V_\lambda\cong (\det)^k\otimes V_\mu.
	$
	Thus the projective equivalence class depends only on~$\mu$, not on $k$.
	
	\begin{corollary}\label{app:cor-projective-gap}
		Let $\varphi:\U(n)\to \U(m)$ be an irreducible representation. Assume $m<\binom{n}{2}$. Then one of the following holds:
		\begin{enumerate}
			\item[(i)] $m=1$, in which case $\varphi(v)=\det(v)^k$ for some $k\in \bbZ$.
			\item[(ii)] $m=n$, and up to multiplication by a determinant character, $\varphi$ is either the standard representation or the dual representation. More precisely, there exist $k\in \bbZ$ and $u\in \U(n)$ such that either
			\[
			\varphi(v)=\det(v)^k\,u v u^{-1}
			\quad\mbox{or}\quad
			\varphi(v)=\det(v)^k\,u\overline{v}\,u^{-1}
			\qquad (v\in \U(n)).
			\]
		\end{enumerate}
	\end{corollary}
	
	\begin{proof}
		Write the highest weight as $\lambda$, and then
		$
		V_\lambda\cong (\det)^k\otimes V_\mu,$ $\mu_n=0.$
		
		If $m=1$, then $V_\mu$ is trivial, so $\varphi=(\det)^k$.
		If $m>1$, then $V_\mu$ is nontrivial and $1<\dim V_\mu<\binom{n}{2}$. By the classical classification of low-dimensional irreducible representations for type $A_{n-1}$, the only such irreducible $\SU(n)$-modules are the standard representation and its dual; see, for example, \cite{FH} and \cite{Weyl}. Under our normalization $\mu_n=0$, these correspond respectively to
		\[
		\mu=(1,0,\dots,0)
		\qquad\text{and}\qquad
		\mu=(\underbrace{1,\dots,1}_{n-1},0).
		\]
		Hence $V_\mu$ is either the standard representation or $\Alt^{n-1}(\bbC^n)$. Moreover,
		$
		\Alt^{n-1}(\bbC^n)\cong \det\otimes (\bbC^n)^*.
		$
		Therefore, up to a determinant twist, $\varphi$ is either the standard or the dual representation.
	\end{proof}
	
	\begin{lemma}\label{app:lem-projective-metric-compare}
		For every $N\ge 1$ and all $x,y\in \PU(N)$,
		\[
		\frac{2}{\pi}\,\delta_N(x,y)\le \bar d_N(x,y)\le \delta_N(x,y).
		\]
	\end{lemma}
	
	\begin{proof}
		By bi-invariance it suffices to consider $x=\bar I$ and $y=\bar u$.
		For each $\omega\in \mathbb{T}$, write the eigenvalues of $\omega u$ as $e^{i\theta_1},\dots,e^{i\theta_N}$ with $\theta_j\in [-\pi,\pi]$. Then
		\[
		\dHSk{N}(I,\omega u)^2=\frac1N\sum_{j=1}^N 4\sin^2\!\Bigl(\frac{\theta_j}{2}\Bigr),
		\]
		whereas the geodesic distance on $\U(N)$ is
		\[
		\dist_{\mathrm{geo}}(I,\omega u)^2=\frac1N\sum_{j=1}^N \theta_j^2.
		\]
		Since for $|t|\le \pi$ we have 
		$
		\frac{2}{\pi}|t|\le 2\sin\Bigl(\frac{|t|}{2}\Bigr)\le |t|$,
		we get
		\[
		\frac{2}{\pi}\,\dist_{\mathrm{geo}}(I,\omega u)
		\le \dHSk{N}(I,\omega u)
		\le \dist_{\mathrm{geo}}(I,\omega u)
		\qquad (\omega\in \mathbb{T}).
		\]
		Taking the infimum over $\omega\in \mathbb{T}$ gives the claim.
	\end{proof}
	
	For a compact metric space $(X,d)$, we denote by
	$
	\mathcal N(X,d,r)
	$
	the smallest cardinality of an $r$-net in $X$.
	
	\begin{lemma}\label{app:lem-entropy-lower}
		There exists an absolute constant $r_0>0$ such that for every fixed $0<r<r_0$ there exists $c_r>0$ with
		\[
		\log \mathcal N\bigl(\PU(n),\bar d_n,r\bigr)\ge c_r(n^2-1)
		\qquad (n\ge 2).
		\]
	\end{lemma}
	
	\begin{proof}
		By Lemma~\ref{app:lem-projective-metric-compare},
		\[
		\mathcal N\bigl(\PU(n),\bar d_n,r\bigr)
		\ge
		\mathcal N\Bigl(\PU(n),\delta_n,\frac{\pi r}{2}\Bigr).
		\]
		Let $\vol_n$ denote the Riemannian volume associated with $\delta_n$.
		By homogeneity, we have
		\[
		\mathcal N\bigl(\PU(n),\delta_n,r\bigr)
		\ge
		\frac{\vol_n(\PU(n))}{\vol_n(B(e,r))}.
		\]
		
		It is well-known that the Ricci curvature of $\PU(n)$ with respect to $\delta_n$ is nonnegative. By the Bishop--Gromov comparison theorem (see, for example, \cite{Petersen}*{Chapter~9}), for every $0<r<\infty$,
		$
		\vol_n(B(e,r))\le \omega_{n^2-1}\, r^{n^2-1},
		$
		where $\omega_{n^2-1}$ is the Euclidean volume of the unit ball in $\bbR^{n^2-1}$:
		\[
		\omega_{n^2-1}=\frac{\pi^{(n^2-1)/2}}{\Gamma((n^2-1)/2+1)}.
		\]
		Therefore
		\begin{equation}\label{app:eq-covering-lower}
			\mathcal N\bigl(\PU(n),\delta_n,r\bigr)
			\ge
			\frac{\vol_n(\PU(n))}{\omega_{n^2-1}\, r^{n^2-1}}.
		\end{equation}
		The classical volume formula for $\U(n)$ with respect to the unnormalized trace metric is
		\begin{equation}\label{app:eq-u-volume-hs}
			\vol\bigl(\U(n),g_n^{\mathrm{HS}}\bigr)
			=
			\frac{(2\pi)^{n(n+1)/2}}{\prod_{j=1}^{n-1} j!};
		\end{equation}
		see, for example, Macdonald's volume formula for compact Lie groups. Now, for the normalized trace metric $g_n$, we have an additional factor of $n^{-n^2/2}$. Hence
		\begin{equation}\label{app:eq-u-volume-normalized}
			\vol\bigl(\U(n),g_n\bigr)
			=
			n^{-n^2/2}\,
			\frac{(2\pi)^{n(n+1)/2}}{\prod_{j=1}^{n-1} j!}.
		\end{equation}
		
		The quotient map $q:\U(n)\to \PU(n)=\U(n)/\mathbb{T}$
		is a Riemannian submersion, because the scalar circle is central and totally geodesic. The fiber is $\mathbb{T}=\{e^{it}I_n:\ t\in \bbR/2\pi\bbZ\}$ and has length $2\pi$. It follows that
		\begin{equation}\label{app:eq-pu-volume}
			\vol_n\bigl(\PU(n)\bigr)
			=
			\frac{1}{2\pi}\,
			\vol\bigl(\U(n),g_n\bigr)
			=
			n^{-n^2/2}\,
			\frac{(2\pi)^{(n^2+n-2)/2}}{\prod_{j=1}^{n-1} j!}.
		\end{equation}
		
		Now, since $j! \le j^j$, we have $\sum_{j=1}^{n-1} \log(j!) \le \frac{n^2}{2}\log(n)$. Therefore,
		\[
		\log \vol_n\bigl(\PU(n)\bigr)
		\ge
		-n^2 \log n+\frac{n^2}2 \log(2\pi).
		\]
		Using
		$
		\omega_{n^2-1}=\frac{\pi^{(n^2-1)/2}}{\Gamma((n^2-1)/2+1)},
		$
		Stirling's formula yields
		\[
		-\log \omega_{n^2-1}
		\ge
		\frac{n^2-1}{2}\log(n^2-1)-\frac{n^2-1}{2} \log(2\pi e).
		\]
		Combining this with \eqref{app:eq-covering-lower} and \eqref{app:eq-pu-volume}, we obtain
		\[
		\log \mathcal N\bigl(\PU(n),\delta_n,r\bigr)
		\ge
		\log \vol_n\bigl(\PU(n)\bigr)-\log \omega_{n^2-1}-(n^2-1)\log r
		\]
		and therefore
		\[
		\begin{aligned}
			\log \mathcal N\bigl(\PU(n),\delta_n,r\bigr)
			\ge{}& -n^2 \log n + \frac{n^2}2 \log(2\pi) +\frac{n^2-1}{2}\log(n^2-1)\\
			& -\frac{n^2-1}{2} \log(2 \pi e)  -(n^2-1)\log r.
		\end{aligned}
		\]
		We now estimate the $n$-dependent terms. First,
		\[
		\frac{n^2-1}{2}\log(n^2-1)-n^2\log n
		=
		-\log n+\frac{n^2-1}{2}\log\Bigl(1-\frac{1}{n^2}\Bigr).
		\]
		Since $n\ge 2$, we have $1-1/n^2\ge 3/4$, and therefore
		\[
		\frac{n^2-1}{2}\log(n^2-1)-n^2\log n
		\ge
		-\log n+\frac{n^2-1}{2}\log(3/4).
		\]
		Also $\log n\le n^2-1$ for $n\ge 2$, so the right-hand side is bounded below by $-c(n^2-1)$ for some absolute constant $c>0$. Next,
		\[
		\frac{n^2}2\log(2\pi)-\frac{n^2-1}{2}\log(2\pi e)
		= \frac12 \log(2\pi) - \frac{n^2-1}2 \log(e) \geq - \frac{n^2-1}{2}.
		\]
		Combining these estimates, we get a constant $c>0$ and
		\[
		\log \mathcal N\bigl(\PU(n),\delta_n,r\bigr)
		\ge
		(n^2-1)\log(1/r)-c(n^2-1)
		\]
		for every $r>0$. Now let $y\in \PU(n)$. By Lemma~\ref{app:lem-projective-metric-compare}, if $\bar d_n(x,y)<r$, then
		$
		\delta_n(x,y)\le \frac{\pi}{2}\,\bar d_n(x,y)<\frac{\pi r}{2}.
		$
		Thus every $\bar d_n$-ball of radius $r$ is contained in the $\delta_n$-ball of radius $\pi r/2$ with the same center. Consequently, every $r$-net for $(\PU(n),\bar d_n)$ is a $\pi r/2$-net for $(\PU(n),\delta_n)$, and hence
		\[
		\log \mathcal N\bigl(\PU(n),\bar d_n,r\bigr)
		\ge
		(n^2-1)\log\Bigl(\frac{2}{\pi r}\Bigr)-c(n^2-1).
		\]
		Set
		$
		r_0:=\min\Bigl\{1,\,\frac{2}{\pi}e^{-c}\Bigr\}.
		$
		Then for every fixed $0<r<r_0$, the constant
		$
		c_r:=\log\Bigl(\frac{2}{\pi r}\Bigr)-c
		$
		is positive, and therefore $\log \mathcal N\bigl(\PU(n),\bar d_n,r\bigr)\ge c_r(n^2-1)$.
	\end{proof}
	
	\begin{lemma}\label{app:lem-covering-upper}
		Let $\rho:\U(n)\to \U(m)$ be an irreducible representation whose induced projective representation is non-trivial, and let $\bar\rho:\PU(n)\to \PU(m)$
		be the induced projective representation.
		Then for every $0<r<1$,
		\[
		\log \mathcal N\bigl(\bar\rho(\PU(n)),\bar \dmet_m,r\bigr)
		\le
		(n^2-1)\log\Bigl(\frac{4\pi m^2}{r}\Bigr).
		\]
	\end{lemma}
	
	\begin{proof}
		Let $\lambda=(\lambda_1,\dots,\lambda_n)$ be the highest weight of $\rho$, and set
		$
		a:=\lambda_1-\lambda_n\in \bbN.
		$
		The Lie algebra of $\PU(n)$ is $\mathfrak{su}(n)$, equipped with the norm
		\[
		\|X\|_{n}:=\Bigl(\frac1n\tr(X^*X)\Bigr)^{1/2}.
		\]
		Because $\bar\rho$ is a non-trivial homomorphism from the simple compact Lie group $\PU(n)$, the pull-back of the normalized trace metric from $\PU(m)$ to $\PU(n)$ is a bi-invariant metric obtained from $\|\cdot\|_n$ by multiplying with a scalar $\alpha=\alpha(\rho)$:
		$
		\|\dmet\rho(X)\|_{m}=\alpha\,\|X\|_{n}$ 
		(for $X\in \mathfrak{su}(n)$).
		Let
		$
		H:=i(E_{11}-E_{nn})\in \mathfrak{su}(n).
		$
		Then
		\[
		\|H\|_n=\Bigl(\frac1n\tr(H^*H)\Bigr)^{1/2}=\sqrt{\frac{2}{n}}.
		\]
		By standard highest-weight theory for the embedded $\SU(2)$-subgroup corresponding to the highest root $e_1-e_n$, every weight of the restriction of $\rho$ to this subgroup lies in $\{-a,-a+2,\dots,a-2,a\}$; see \cite{KnappLieGroups}. Hence every eigenvalue of $\dmet\rho(H)$ has absolute value at most $a$, and therefore $\|\dmet\rho(H)\|_m\le a$. Since $\|\dmet\rho(H)\|_m=\alpha\,\|H\|_n$, it follows that
		$
		\alpha\le a\sqrt{\frac{n}{2}}.
		$
		
		Again by standard highest-weight theory, the restriction of $\rho$ to the highest-root $\SU(2)$-subgroup contains an irreducible $\SU(2)$-subrepresentation of dimension $a+1$, so $a+1\le m$; see \cite{KnappLieGroups}. Moreover, the smallest non-trivial projective irreducible representation of $\SU(n)$ has dimension $n$, so $n\le m$; see \cite{FH} and \cite{Weyl}. Consequently,
		$
		\alpha\le (m-1)\sqrt{\frac{m}{2}}\le m^2.
		$
		
		Let $\delta_{\bar\rho(\PU(n))}$ be the quotient geodesic distance on $\bar\rho(\PU(n))$ induced from $\delta_m$. Then $\bar\rho:\PU(n)\to \bar\rho(\PU(n))$ is a local homothety of factor $\alpha$, so $\bar\rho(\PU(n))$ is a compact homogeneous Riemannian manifold of real dimension $n^2-1$, with nonnegative Ricci curvature.
		
		By Lemma~\ref{app:lem-projective-metric-compare}, it suffices to cover $\bar\rho(\PU(n))$ with respect to $\delta_{\bar\rho(\PU(n))}$, because $\bar \dmet_m\le \delta_{\bar\rho(\PU(n))}$ on $\bar\rho(\PU(n))$. Also $\diam(\PU(n),\delta_n)\le \pi$: indeed, every class in $\PU(n)$ has a representative $u\in \SU(n)$ of the form $u=e^{iH}$ with $H=H^*$, $\tr(H)=0$, and all eigenvalues of $H$ in $[-\pi,\pi]$, so $\delta_n(\bar I,\bar u)=\|H\|_n\le \pi$. Therefore
		$\diam\bigl(\bar\rho(\PU(n)),\delta_{\bar\rho(\PU(n))}\bigr)\le \pi\alpha$.
		We use the following form of the Bishop--Gromov comparison theorem (see, for example, \cite{Petersen}*{Chapter~9}): if $M$ is a complete $d$-dimensional Riemannian manifold with nonnegative Ricci curvature, then for every $p\in M$ and all $0<s<t$,
		\[
		\frac{\vol(B(p,t))}{\vol(B(p,s))}\le \Bigl(\frac{t}{s}\Bigr)^d.
		\]
		
		Take a maximal $r/2$-separated subset of $\bar\rho(\PU(n))$ in the metric $\delta_{\bar\rho(\PU(n))}$, say $y_1,\dots,y_N$. Then the balls $B(y_j,r/4)$ are disjoint and all lie inside $B(y_1,\pi\alpha)$. Bishop--Gromov comparison for nonnegative Ricci curvature gives
		\[
		N\le \frac{\vol(B(y_1,\pi\alpha))}{\vol(B(y_1,r/4))}
		\le \Bigl(\frac{\pi\alpha}{r/4}\Bigr)^{n^2-1}
		\le \Bigl(\frac{4\pi m^2}{r}\Bigr)^{n^2-1},
		\]
		using $\alpha\le m^2$ in the last step. Since a maximal $r/2$-separated set is an $r/2$-net, it is also an $r$-net for $\bar \dmet_m$. Hence
		\[
		\log \mathcal N\bigl(\bar\rho(\PU(n)),\bar \dmet_m,r\bigr)
		\le
		(n^2-1)\log\Bigl(\frac{4\pi m^2}{r}\Bigr).
		\]
		This finishes the proof.
	\end{proof}

	\begin{lemma}\label{app:lem-irreducible-reduction}
		Fix $\eta\in (0,1/3)$. Then there exists $\delta_\eta>0$ with the following property.
		If $\varphi:\U(n)\to \U(m)$	is $\delta$-surjective with $0<\delta<\delta_\eta$, then among the irreducible summands of $\varphi$ there is one, say $\rho:\U(n)\to \U(m_0)$, with $m_0>(1-\eta)m$.
		Moreover, the induced projective representation
		$
		\bar\rho:\PU(n)\to \PU(m_0)
		$
		is $\delta/\sqrt{1-\eta}$-surjective.
	\end{lemma}
	
	\begin{proof}
		If after unitary conjugacy $\varphi=\psi\oplus \chi$ with
		$
		\dim \psi=m_1\ge \eta m,$
		and
		$		\dim \chi=m_2\ge \eta m,
		$
		then $\varphi$ is not $\delta$-surjective for any $\delta<2\sqrt\eta$.
		Indeed, write $\bbC^m=E\oplus F$ accordingly. Choose subspaces $E_0\subseteq E$ and $F_0\subseteq F$ of common dimension $q:=\min(m_1,m_2)\ge \eta m$. 
		Let $u_0\in \U(m)$ act by swapping $E_0$ and $F_0$, and by the identity on the orthogonal complement. For any $g\in {\U}(n)$, the matrix $\varphi(g)$ preserves $E$ and $F$, so we can compute
		$\|u_0-\varphi(g)\|_{\mathrm{HS}}^2\ge 4q$.
		Therefore $ \dHSk{m}( u_0,\varphi(  g))\ge 2\sqrt{\frac{q}{m}}\ge 2\sqrt{\eta}.$
		
		After shrinking $\delta_\eta$ if necessary, we may therefore assume that in every direct sum decomposition $\varphi=\psi\oplus \chi$ one of the two summands has dimension $<\eta m$.
		
		Now, in a decomposition into irreducibles, one irreducible summand has dimension $>(1-\eta)m$.
		Write $\varphi=\rho\oplus \sigma$ with $\rho$ irreducible and $\dim \rho=m_0>(1-\eta)m$. We now show that $\bar\rho$ is $\delta/\sqrt{1-\eta}$-surjective.
		Recall that $\varphi$ is $\delta$-surjective on $\U(m)$. Fix $v\in \U(m_0)$ and consider $x:=v\oplus I_{m-m_0}\in \U(m)$.
		Choose $g\in \U(n)$ such that
		$\dHSk{m}(x,\varphi(g))\le \delta$. Since $\varphi(g)=\rho(g)\oplus \sigma(g)$, we get
		$
		\|v-\rho(g)\|_{\mathrm{HS}}^2
		\le
		\|x-\varphi(g)\|_{\mathrm{HS}}^2
		\le
		m\delta^2.
		$
		Hence
		$
		\bar \dmet_{m_0}(\bar v,\bar\rho(\bar g))
		\le
		\Bigl(\frac{m}{m_0}\Bigr)^{1/2}\delta
		<
		\frac{\delta}{\sqrt{1-\eta}}.
		$
		Thus $\bar\rho$ is $\delta/\sqrt{1-\eta}$-surjective.
	\end{proof}
	
	\begin{proposition}\label{app:prop-coarse-bound}
		There exists $\delta_0>0$ and a constant $A>0$ such that whenever
		$
		\varphi:\U(n)\to \U(m)
		$
		is $\delta$-surjective with $0<\delta<\delta_0$, one has
		$
		m^2 \le A\,n^2\log(2m).
		$
	\end{proposition}
	
	\begin{proof}
		If $m=1$, then
		$
		m^2=1\le \frac{1}{\log 2}\,n^2\log(2m).
		$
		Thus, after enlarging $A$ if necessary, the conclusion is automatic in this case. We may therefore assume that $m>1$.
		
		Fix $\eta=1/4$. By Lemma~\ref{app:lem-irreducible-reduction}, after shrinking $\delta_0$ if necessary there is an irreducible summand
		$
		\rho:\U(n)\to \U(m_0)
		$
		with
		$
		m_0>\frac34 m,
		$
		such that $\bar\rho:\PU(n)\to \PU(m_0)$ is $(2/\sqrt3)\delta$-surjective.
		Let $r_0>0$ be given by Lemma~\ref{app:lem-entropy-lower}, and fix
		$
		s_0:=\min\Bigl\{\frac12,\frac{r_0}{2}\Bigr\}.
		$
		Then $0<s_0<r_0$. Shrink $\delta_0$ further so that
		$
		\frac{2\delta_0}{\sqrt3}<\frac{s_0}{2}.
		$
		Now fix $0<\delta<\delta_0$. Since $\bar\rho$ is $(2/\sqrt3)\delta$-surjective, every point of $\PU(m_0)$ lies at $\bar \dmet_{m_0}$-distance less than $s_0/2$ from $\bar\rho(\PU(n))$. Therefore any $s_0/2$-net in $\bar\rho(\PU(n))$ is an $s_0$-net in $\PU(m_0)$. Hence
		\[
		\mathcal N\bigl(\PU(m_0),\bar d_{m_0},s_0\bigr)
		\le
		\mathcal N\bigl(\bar\rho(\PU(n)),\bar \dmet_{m_0},s_0/2\bigr).
		\]
		By Lemma~\ref{app:lem-entropy-lower}, there exists a constant $c_{s_0}>0$ such that
		\[
		\log \mathcal N\bigl(\PU(m_0),\bar \dmet_{m_0},s_0\bigr)
		\ge
		c_{s_0}\bigl(m_0^2-1\bigr).
		\]
		Applying also Lemma~\ref{app:lem-covering-upper} with $r=s_0/2$, we get
		\[
		c_{s_0}\bigl(m_0^2-1\bigr)
		\le
		\log \mathcal N\bigl(\bar\rho(\PU(n)),\bar d_{m_0},s_0/2\bigr)
		\le
		(n^2-1)\log\Bigl(\frac{8 \pi m_0^2}{s_0}\Bigr).
		\]
		Since $m_0>\frac34 m$, the left-hand side is bounded below by a positive constant multiple of $m^2$. Since also $m_0\le m$, the right-hand side is bounded above by a constant multiple of $n^2\log(2m)$. Hence, after enlarging constants, we obtain
		$
		m^2 \le A\,n^2\log(2m),
		$
		as claimed.
	\end{proof}
	
	\begin{proof}[Proof of Theorem~\ref{th_stab}]
		Replacing $\varepsilon$ by $\min\{\varepsilon,1\}$ if necessary, we may assume $0<\varepsilon<2$. Fix such an $\varepsilon$, and set $\eta:= \varepsilon^2/16$. Let $\delta_0$ and $A$ be given by Proposition~\ref{app:prop-coarse-bound}, and let $\delta_\eta$ be given by Lemma~\ref{app:lem-irreducible-reduction}. Choose
		$
		\delta<\min\Bigl\{\delta_0,\, (\sqrt3/2)\delta_\eta\Bigr\}.
		$
		Let $\varphi:{\U}(n)\to {\U}(m)$ be $\delta$-surjective.
		
		If $m=1$, then $\varphi$ is a character, hence $\varphi(v)=\det(v)^k$ for some $k\in \bbZ$.
		Assume now that $m>1$. By Proposition~\ref{app:prop-coarse-bound}, $m^2\le A\,n^2\log(2m)$.
		Hence there exists $N\in \bbN$ such that for every $n\ge N$, the inequality
		$
		m\ge \binom{n}{2}
		$
		is impossible. Assume first that $n\ge N$. Then $m<\binom{n}{2}$. Since $\delta<(\sqrt3/2)\delta_\eta$, Lemma~\ref{app:lem-irreducible-reduction} yields an irreducible summand $\rho:\U(n)\to \U(m_0)$ with $m_0>(1-\eta)m$, such that the induced projective representation $\bar\rho:\PU(n)\to \PU(m_0)$ is $\delta/\sqrt{1-\eta}$-surjective. In particular $m_0>1$; indeed, since $m>1$ and $\eta\le 1/4$, we have
		$
		m_0>(1-\eta)m\ge \frac34\cdot 2>1.
		$
		Since also $m_0\le m<\binom{n}{2}$, Corollary~\ref{app:cor-projective-gap} shows that $m_0=n$ and that, up to unitary conjugacy and multiplication by a determinant character, $\rho$ is either the standard representation or the dual representation.
		
		After conjugating $\varphi$, write $\varphi=\rho\oplus \sigma$.
		Then $n=m_0>(1-\eta)m$, so in particular $n\le m\le \frac{n}{1-\eta}\le (1+\varepsilon)n$.
		Set $q:=m-n$. Then $q=m-n<m\eta$. If $\rho(v)=\det(v)^k\,u_0vu_0^{-1}$, let $u:=u_0\oplus I_q$. Then
		$
		\varphi(v)=\det(v)^k\,\ad_u\bigl(v\oplus \det(v)^{-k}\sigma(v)\bigr),
		$
		and therefore
		\[
		\begin{aligned}
			\dHSk{m}\bigl(\varphi(v),\det(v)^k\,\ad_u(v\oplus I_q)\bigr)
			&= \dHSk{m}\bigl(v\oplus \det(v)^{-k}\sigma(v),v\oplus I_q\bigr) \\
			&= \dHSk{m}\bigl(\det(v)^{-k}\sigma(v),I_q\bigr) \\
			&\le 2\sqrt{\frac{q}{m}}
			<2\sqrt\eta\le \varepsilon.
		\end{aligned}
		\]
		The case with $\overline{v}$ in place of $v$ is identical. This proves the theorem when $n\ge N$.
		
		It remains to treat the finitely many values $1\le n<N$. By
		Proposition~\ref{app:prop-coarse-bound}, there exists $M\in\bbN$ such that
		every $\delta$-surjective homomorphism $\U(n)\to\U(m)$ with $1\le n<N$
		satisfies $m\le M$.
		
		For fixed $1\le n<N$ and $1<m_0\le M$, consider the projective
		homomorphisms
		$
		\tau\colon \PU(n)\to\PU(m_0)
		$
		arising from irreducible continuous representations
		$
		\rho\colon \U(n)\to\U(m_0).
		$
		Only finitely many such projective types occur. Call such a type good if
		$\rho$ is, up to unitary conjugacy and multiplication by a determinant
		character, either the standard representation or the dual representation.
		
		For every bad projective type $\tau$, set
		\[
		\Delta_\tau:=
		\sup_{y\in\PU(m_0)}
		\bar\dmet_{m_0}(y,\tau(\PU(n))).
		\]
		As before, $\Delta_\tau>0$. Since only finitely many bad types occur, the
		minimum $\Delta$ of these numbers is positive. Shrink $\delta$ once more so
		that $\frac{\delta}{\sqrt{1-\eta}}<\Delta.$
		
		Now let $\varphi\colon\U(n)\to\U(m)$ be $\delta$-surjective with
		$1\le n<N$ and $1<m\le M$. By Lemma~\ref{app:lem-irreducible-reduction},
		there is an irreducible summand
		$
		\rho\colon\U(n)\to\U(m_0)
		$
		with $m_0>(1-\eta)m$ such that
		$
		\bar\rho\colon\PU(n)\to\PU(m_0)
		$
		is $\delta/\sqrt{1-\eta}$-surjective. Hence the projective class of $\rho$
		cannot be bad. Therefore $\rho$ is, up to unitary conjugacy and determinant
		twist, either the standard representation or the dual representation. The rest
		of the proof is identical to the case $n\ge N$.
		
		If in addition \(\varphi(zI_n)=zI_m\) for all \(z\in\mathbb T\), then every
		irreducible summand of \(\varphi\) has central character \(z\). For a standard
		summand twisted by \(\det^k\), the central character is \(z^{kn+1}\), hence
		\(kn+1=1\), so \(k=0\). For a dual summand twisted by \(\det^k\), the central
		character is \(z^{kn-1}\), hence \(kn-1=1\). This again gives \(k=0\), except
		for the harmless case $n=2$, $k=1$, where \(\det\otimes \overline{\mathrm{std}}\)
		is equivalent to the standard representation. Thus the approximation may
		be written with \(k=0\).
	\end{proof}
	
	\section{Metric reduced products of unitary groups}\label{sec:reduced-products}
	
	\subsection{Preliminaries and definitions}\label{sec:preliminaries}
	
	Recall that when $(k_n)_n$ is a sequence of natural numbers, we denote by $\UHS[ (k_n)_n]$ the metric reduced product of the sequence of metric groups $(\U(k_n): n \in \bbN)$ in the normalized trace norm.
	
	Given a sequence $(u_n)_n$ with $u_n\in {\U}(k_n)$, we write $[u_n]_n$ for the element of $\UHS[(k_n)_n]$ that it represents. We also consider $\MHS[(k_n)_n]$, the $\| \|_{2,n}$-reduced $\| \|_{\text{op}}$-$\ell_\infty$-product of the $\M_{k_n}(\mathbb C)$. Namely, $\MHS[(k_n)_n]$ has as elements $[u_n]_n$ for $(u_n)_n$ a sequence with $u_n\in \M_{k_n}(\mathbb C)$ that is bounded in operator norm. For two such sequences $(u_n)_n$ and $(v_n)_n$, we set $[u_n]_n=[v_n]_n$ if
	$
	\lim_n \dHSk{k_n}(u_n,v_n)=0,
	$
	and the distance between two elements $[u_n]_n,[v_n]_n\in\MHS[(k_n)_n]$ is given by
	\[
	\dHSinf\left([u_n]_n,[v_n]_n\right)=\limsup_n \dHSk{k_n}(u_n,v_n).
	\]
	Clearly, $\UHS[ (k_n)_n]$ naturally embeds into $\MHS[(k_n)_n]$.
	The coordinatewise algebraic operations and involution descend to
	$\MHS[(k_n)_n]$, so $\MHS[(k_n)_n]$ is a unital $*$-algebra.
	
	If $(k_n)_n$ and $(l_n)_n$ are asymptotically equivalent, that is, $\lim \frac{l_n}{k_n}=1$, then $\UHS[ (k_n)_n]$ and $\UHS[ (l_n)_n]$ are canonically isomorphic. When $k_n\le l_n$ for all $n$, every element $[v_n]_n\in \UHS[ (l_n)_n]$ determines a unique element $[u_n]_n\in\UHS[ (k_n)_n]$ with
	\[
	\dHSinf([u_n]_n,[v_n\upharpoonright \bbC^{k_n}]_n)=0,
	\]
	and we denote it by $[v_n]_n \downarrow \UHS[ (k_n)_n]$. Conversely, every $[u_n]_n\in\UHS[ (k_n)_n]$ admits a unique lift $[v_n]_n\in\UHS[ (l_n)_n]$ such that $[u_n]_n=[v_n]_n\downarrow \UHS[ (k_n)_n]$; we denote it by $[u_n]_n\uparrow \UHS[ (l_n)_n]$. In general, when $(k_n)_n$ and $(l_n)_n$ are asymptotically equivalent, there is a unique isometric group isomorphism
	$
	\rho\colon \UHS[ (k_n)_n]\to\UHS[ (l_n)_n]
	$
	such that for every $[u_n]_n\in\UHS[ (k_n)_n]$ and $[v_n]_n\in\UHS[ (l_n)_n]$:
	\begin{align*}
		\rho([u_n]_n)=[v_n]_n&\Leftrightarrow
		[u_n]_n\downarrow \UHS[ (k_n\wedge l_n)_n]=[v_n]_n\downarrow \UHS[ (k_n\wedge l_n)_n]\\
		&\Leftrightarrow
		[u_n]_n\uparrow \UHS[ (k_n\vee l_n)_n]=[v_n]_n\uparrow \UHS[ (k_n\vee l_n)_n]\\
		&\Leftrightarrow
		\lim_n
		\dHSk{k_n\wedge l_n}
		\bigl(u_n\upharpoonright \bbC^{k_n\wedge l_n},v_n\upharpoonright \bbC^{k_n\wedge l_n}\bigr)=0\\
		&\Leftrightarrow
		\lim_n
		\dHSk{k_n\vee l_n}
		\bigl(u_n\oplus I_{k_n\vee l_n-k_n},v_n\oplus I_{k_n\vee l_n-l_n}\bigr)=0.
	\end{align*}
	
	We denote the image of $[u_n]_n$ under this isomorphism by $[u_n]_n\updownarrow \UHS[ (l_n)_n]$.
	
	\begin{definition}\label{respect convergence}
		A sequence $(h_n)_n$ of maps $h_n\colon \U(k_n)\to \U(l_n)$ is said to \emph{respect $\dHS$-convergence} if for all sequences $(u_n)_n,(v_n)_n$ with $u_n,v_n\in \U(k_n)$ for all~$n$ one has
		\[
		[u_n]_n=[v_n]_n \Rightarrow [h_n(u_n)]_n=[h_n(v_n)]_n.
		\]
	\end{definition}
	
	\begin{definition}
		A map $\varphi\colon\UHS[(k_n)_n]\to\UHS[(l_n)_n]$ is said to be \emph{of product form} if there exists a sequence $(h_n)_n$ of maps $h_n\colon \U(k_n)\to \U(l_n)$ such that for every sequence $(u_n)_n$ with $u_n\in \U(k_n)$ the map $\varphi$ sends $[u_n]_n\in\UHS[(k_n)_n]$ to $[h_n(u_n)]_n\in\UHS[(l_n)_n]$. Note that then the sequence $(h_n)_n$ necessarily respects $\dHS$-convergence.
	\end{definition}
	
	A natural first question arises: if  $\UHS[ (k_n)_n] \cong \UHS[ (l_n)_n],$ must $(k_n)_n$ and $(l_n)_n$ be asymptotically equivalent?    We first answer this question when the isomorphism can be taken to be \emph{of product form}.
	
	Every sequence $(\alpha_n)_n$ with $\alpha_n \in \Aut(\U(k_n))$ respects $\dHS$-convergence (all automorphisms of $\U(k)$ are isometric) and hence induces an automorphism of the group $\UHS[ (k_n)_n]$ which is of product form.
	By the following proposition, all product form automorphisms of the group $\UHS[ (k_n)_n]$ can be obtained in that way.
	
	\begin{proposition} \label{prop_stab}
		Let $(k_n)_n$ and $(l_n)_n$ be two sequences of natural numbers with $\lim_n k_n = \lim_n l_n = \infty.$
		Let $\varphi : \UHS[ (k_n)_n] \to \UHS[ (l_n)_n]$ be an isomorphism of product form. Then $\lim \frac{k_n}{l_n}=1$ and there exists a sequence $(\alpha_n)_n$ with $\alpha_n \in \Aut(\U(l_n))$ such that $\varphi$ maps every $[u_n]_n \in \UHS[(k_n)_n]$ to $[\alpha_n(v_n)]_n \in \UHS[(l_n)_n]$, where $[v_n]_n = [u_n]_n\updownarrow \UHS[ (l_n)_n].$
	\end{proposition}
	
	Subsection~\ref{subsec: prop_stab} is devoted to the proof of Proposition~\ref{prop_stab}.
	
	The next result drops the product-form assumption and uses a different mechanism: the definability of the absolute normalized trace on involutions together with uniform bounded normal generation modulo the center. It expresses that the reduced product groups $\UHS[(k_n)_n]$ \emph{recognise coordinates}.
	
	\begin{theorem}\label{th.action boundary}
		Every isomorphism $\varphi: \UHS[(k_n)_n] \to \UHS[(l_n)_n]$ induces an automorphism $\theta$ of the Stone--\v{C}ech boundary $\partial \beta \bbN,$ which moreover has the following two properties:
		\begin{enumerate}
			\item\label{itm1} for all $a,b \in \UHS[(k_n)_n]$ and any infinite subset $S \subseteq \bbN,$ we have \[a =_S b \Leftrightarrow \varphi(a) =_{\theta(S)} \varphi(b),\]
			\item\label{itm2} for every nonprincipal ultrafilter $\cU \in \partial \beta \bbN$, the isomorphism $\varphi$ induces an isomorphism $\varphi_\cU$ between the metric ultraproduct groups $\prod_n  {\U}(k_n) / \cU$ and\linebreak $\prod_n  {\U}(l_n) / \theta^{-1}(\cU)$.
		\end{enumerate}
	\end{theorem}
	
	We can now apply Lemma~\ref{lem.rigidGeneral} to derive from    Proposition~\ref{prop_stab} and Theorem~\ref{th.action boundary} the following main rigidity theorem for isomorphisms between the groups $\UHS[(k_n)_n]$.
	
	\begin{theorem}\label{th.rigunitary}\
		Assuming $\mathrm{OCA}+\mathrm{MA}_{\aleph_1}(\sigma\text{-linked})$, every group isomorphism
		\[    \varphi:\UHS[(k_n)_n]\to\UHS[(l_n)_n]\]
		is trivial, i.e.\ there exists an almost permutation $f : \bbN \to \bbN$ such that  $\lim \frac{l_n}{k_{f(n)}}=1$ and there exists a sequence $(\alpha_n)_n$ with $\alpha_n \in \Aut(\U(l_n))$ such that $\varphi$ maps every $[u_n]_n \in \UHS[(k_n)_n]$ to $[\alpha_n(v_n)]_n \in \UHS[(l_n)_n]$, where $[v_n]_n = [u_{f(n)}]_n\updownarrow \UHS[ (l_n)_n].$
		
	\end{theorem}
	
	\subsection{Proof of Proposition \ref{prop_stab}}
	\label{subsec: prop_stab}
	
	\begin{lemma}\label{lem:product-form-asymptotic}
		Let $\varphi:\UHS[(k_n)_n]\to\UHS[(l_n)_n]$ be a product-form isomorphism,
		implemented by a sequence $(h_n)_n$ with $h_n:\U(k_n)\to \U(l_n)$. There are sequences $(L_n)_n$ and $(\delta_n)_n$ with $\delta_n \searrow 0$ such that for every $n$ the map $h_n$ is a
		$\delta_n$-surjective $\delta_n$-homomorphism and is $(L_n,\delta_n)$-quasi-Lipschitz.
	\end{lemma}
	
	\begin{proof}
		For every $\delta>0$, all but finitely many $h_n$ are
		$\delta$-homomorphisms. Otherwise there are $\delta>0$, an infinite set
		$I\subseteq \bbN$, and elements $u_n,v_n\in \U(k_n)$ for $n\in I$ such that
		\[
		\dHSk{l_n}(h_n(u_nv_n),h_n(u_n)h_n(v_n))>\delta
		\qquad (n\in I).
		\]
		Define sequences $(a_n)_n,(b_n)_n$ by $a_n=u_n$, $b_n=v_n$ on $I$ and
		$a_n=b_n=1$ outside $I$. Since $\varphi$ is a homomorphism,
		\[
		[h_n(a_nb_n)]_n=\varphi([a_nb_n]_n)=\varphi([a_n]_n)\varphi([b_n]_n)
		=[h_n(a_n)h_n(b_n)]_n,
		\]
		contradicting the choice of $I$.
		
		Likewise, for every $\delta>0$, all but finitely many $h_n$ are
		$\delta$-surjective. Otherwise there are $\delta>0$, an infinite set
		$I\subseteq \bbN$, and elements $w_n\in \U(l_n)$ for $n\in I$ such that
		$
		\dist(w_n,h_n(\U(k_n)))>\delta.
		$
		Define $(b_n)_n$ by $b_n=w_n$ on $I$ and $b_n=1$ outside $I$. Since
		$\varphi$ is surjective, there exists $[u_n]_n\in \UHS[(k_n)_n]$ with
		$\varphi([u_n]_n)=[b_n]_n$, that is,
		$
		[h_n(u_n)]_n=[b_n]_n,
		$
		again a contradiction.
		
		Finally, fix $\delta>0$. Since the sequence $(h_n)_n$ respects
		$\dHS$-convergence, there are $r(\delta)>0$ and $N(\delta)\in\bbN$ such
		that whenever $n\ge N(\delta)$ and $u,v\in \U(k_n)$ satisfy
		$\dHSk{k_n}(u,v)<r(\delta)$, one has
		$
		\dHSk{l_n}(h_n(u),h_n(v))<\delta.
		$
		Indeed, otherwise one could find an infinite set $I\subseteq \bbN$ and
		pairs $u_n,v_n\in \U(k_n)$ with $\dHSk{k_n}(u_n,v_n)\to 0$ but
		$\dHSk{l_n}(h_n(u_n),h_n(v_n))\ge \delta$ for all $n\in I$, which would
		violate Definition~\ref{respect convergence}. Since the diameter of every
		unitary group in normalized Hilbert--Schmidt distance is at most $2$, it
		follows that for $n\ge N(\delta)$ the map $h_n$ is
		$(2/r(\delta),\delta)$-quasi-Lipschitz.
		
		Thus, we get sequences $(L_n)_n$ and $(\delta_n)_n$ with $\delta_n \to 0$, such that $h_n$ is a $(L_n,\delta_n)$-quasi-Lipschitz $\delta_n$-homomorphism and $\delta_n$-surjective.
	\end{proof}
	
	\begin{lemma}\label{lem:product-form-stabilization}
		Let $\varphi:\UHS[(k_n)_n]\to\UHS[(l_n)_n]$ be a product-form isomorphism,
		implemented by a sequence $(h_n)_n$ satisfying Lemma~\ref{lem:product-form-asymptotic}.
		Then, one has
		$
		\lim_n \frac{k_n}{l_n}=1.
		$
		
		Moreover, there exist integers $a_n\in\bbZ$ and unitaries
		$w_n\in \U(m_n)$ such that the same reduced-product map as $\varphi$ is
		described coordinatewise either by
		\[
		u\mapsto \det(u)^{a_n}\,\Ad_{w_n}(u\oplus I_{m_n-k_n})
		\]
		or by
		\[
		u\mapsto \det(u)^{a_n}\,\Ad_{w_n}(\overline{u}\oplus I_{m_n-k_n}).
		\]
	\end{lemma}
	
	\begin{proof}
		Apply Corollary~\ref{app:cor-dCOT-sequential} to the compact groups
		$
		K_n=\U(k_n)
		$
		and the maps\linebreak $h_n\colon \U(k_n)\to \U(l_n)$. Then, after discarding finitely many indices, there exist integers
		$
		m_n \geq l_n$ with $m_n/l_n\to 1$ and
		continuous homomorphisms
		$
		\rho_n\colon \U(k_n)\to \U(m_n),
		$
		and isometries $W_n\colon \bbC^{l_n}\to \bbC^{m_n}$ such that
		$
		\eta_n:=
		\sup_{u\in\U(k_n)}
		\|h_n(u)-W_n^*\rho_n(u)W_n\|_{2,l_n}
		\to0.
		$

		Let $P_n:=W_nW_n^*$.
		For $v\in {\U}(l_n)$ define
		$
		\widetilde v:=W_n vW_n^*+(1-P_n)\in \U(m_n).
		$
		Then for every $u\in \U(k_n)$,
		\[
		\|\rho_n(u)-\widetilde{h_n(u)}\|_{2,m_n}
		\le \eta_n+2\sqrt{\frac{m_n-l_n}{m_n}}.
		\]
		Conversely, every $y\in \U(m_n)$ is within distance
		$4\sqrt{(m_n-l_n)/m_n}$ of some $\widetilde v$ with $v\in \U(l_n)$: write
		$y$ in block form relative to $P_n$, take the polar unitary of the
		compression $W_n^*yW_n$, and note that the off-diagonal blocks have
		Hilbert--Schmidt norm at most $\sqrt{m_n-l_n}$. Since $h_n$ is
		$\delta_n$-surjective for some sequence $\delta_n \to 0$, it follows that also $\rho_n$ is $\delta'_n$-surjective for
		some sequence $\delta'_n\to 0$.
		
		Fix $\varepsilon>0$. For all sufficiently large $n$, Theorem~\ref{th_stab}
		applied to $\rho_n$ with tolerance $\varepsilon$ gives
		$
		k_n\le m_n\le (1+\varepsilon)k_n.
		$
		The alternative $m_n=1$ is impossible because $m_n/l_n\to 1$ and
		$l_n\to\infty$. Since also $m_n/l_n\to 1$, we conclude that
		$
		\lim_n \frac{k_n}{l_n}=1.
		$
		
		After discarding finitely many indices, Theorem~\ref{th_stab} yields
		integers $a_n\in\bbZ$ and unitaries $w_n\in \U(m_n)$ such that $\rho_n$ is
		$o(1)$-close either to
		\[
		u\mapsto \det(u)^{a_n}\,\Ad_{w_n}(u\oplus I_{m_n-k_n})
		\]
		or to
		\[
		u\mapsto \det(u)^{a_n}\,\Ad_{w_n}(\overline{u}\oplus I_{m_n-k_n}).
		\]
		Because $\eta_n\to0$ and $m_n/l_n\to1$, these formulas describe the same
		reduced-product map as $h_n$ after identifying
		$[u_n]_n\mapsto [u_n]_n\updownarrow \UHS[(l_n)_n]$.
	\end{proof}
	
	\begin{lemma}\label{lem:product-form-det-twist}
		In the situation of Lemma~\ref{lem:product-form-stabilization}, the
		determinant twists vanish eventually. Equivalently, one has $a_n=0$ for all
		sufficiently large $n$.
	\end{lemma}
	
	\begin{proof}
		On scalar matrices the two possibilities from
		Lemma~\ref{lem:product-form-stabilization} act by
		\[
		zI_{l_n}\mapsto z^{a_nk_n+1}I_{l_n}
		\qquad\text{or}\qquad
		zI_{l_n}\mapsto z^{a_nk_n-1}I_{l_n}.
		\]
		For each $n$, let $N_n$ denote the corresponding exponent, so that on
		scalar matrices the induced map is
		\[
		zI_{l_n}\longmapsto z^{N_n}I_{l_n},
		\qquad N_n\in\{a_nk_n+1,a_nk_n-1\}.
		\]
		Since $\varphi$ is injective on the reduced-product center, we claim that
		$|N_n|=1$ for all sufficiently large $n$. Otherwise there is an infinite
		set $I\subseteq \bbN$ such that $|N_n|>1$ for every $n\in I$. For each
		$n\in I$, choose $z_n\in \mathbb T$ with $z_n^{N_n}=1$ and
		$|z_n-1|\ge \sqrt2$: if $N_n$ is even, take $z_n=-1$, and if $N_n$ is
		odd, take
		\[
		z_n:=e^{i\pi(1-1/N_n)}.
		\]
		Then $z_n^{N_n}=1$ and $\Re(z_n)\le 0$, hence $|z_n-1|\ge \sqrt2$.
		Set $z_n:=1$ for $n\notin I$. The element
		$
		z:=[z_nI_{k_n}]_n\in Z(\UHS[(k_n)_n])
		$
		is nontrivial because $|z_n-1|$ stays bounded away from $0$ on the
		infinite set~$I$, whereas its image under $\varphi$ is trivial since
		$z_n^{N_n}=1$ for every $n\in I$. This contradicts injectivity of
		$\varphi$ on the center. Therefore $|N_n|=1$ eventually. Since
		$k_n\to\infty$, this forces $a_n=0$ eventually.
	\end{proof}
	
	\begin{proof}[Proof of Proposition~\ref{prop_stab}]
		Let $(h_n)_n$ be a sequence of maps $h_n:\U(k_n)\to \U(l_n)$ such that
		$
		\varphi([u_n]_n)=[h_n(u_n)]_n$ for all $[u_n]_n\in \UHS[(k_n)_n].$
		By Lemma~\ref{lem:product-form-asymptotic}, after discarding finitely
		many indices the maps $h_n$ satisfy the asymptotic multiplicativity,
		surjectivity, and equicontinuity properties needed for compact-group
		stability. Lemma~\ref{lem:product-form-stabilization} then gives
		$
		\lim_{n\to +\infty} \frac{k_n}{l_n}=1
		$
		and shows that the same reduced-product map as $\varphi$ is described,
		after the canonical identification of asymptotically equivalent
		dimensions, by coordinatewise standard or contragredient automorphisms
		twisted by determinant characters. Lemma~\ref{lem:product-form-det-twist}
		removes the determinant twists. Therefore, after altering finitely many
		coordinates arbitrarily, the map $\varphi$ is induced by a sequence
		$(\alpha_n)_n$ with each $\alpha_n\in \Aut(\U(l_n))$.
	\end{proof}
	
	\subsection{Proof of Theorem \ref{th.action boundary}}
	
	\begin{remark}
		The  center $Z(\UHS[(k_n)_n])$ of the group $\UHS[(k_n)_n]$ consists precisely of those $u\in \UHS[(k_n)_n]$ for which there exists a representing sequence $(u_n)_n \in \prod_n Z({\U}(k_n))$ such that $u = [u_n]_n$ (see \cite{dowerkthom}*{Lemma~7.1} or \cite{dowerk}).
	\end{remark}
	
	We make use of (instances of) the following property of the groups ${\U}(k_n)$, as well as of their \emph{uniform bounded normal generation modulo the center} property as formulated in Lemma~\ref{th.boundedgeneration}.
	
	\begin{lemma} \label{th.stababgroup} {\rm{(see e.g. \cite{Hadwin-Schulman})}}
		If $G$ is a finite abelian group then every homomorphism $\chi : G \to \UHS[(k_n)_n]$ factors through a homomorphism $\widehat{\chi} : G \to \prod_n {\U}(k_n)$.
	\end{lemma}
	
	\begin{definition}
		For $S\in \pow(\bbN)/ \Fin$, define
		\[ Z_S(\UHS[(l_n)_n])=\{ u\in \UHS[(l_n)_n]:(\forall v\in\UHS[(l_n)_n])\  uv=_S vu \}.\]
	\end{definition}
	
	\begin{lemma}\label{th.boundedgeneration}
		For all $\varepsilon > 0$ there exists a natural number $N$ such that for all~$n$ and all $u,v \in \U(n)$ with $\dHSn(u, Z({\U}(n))) \geq \varepsilon$, it is possible to write $v$ as a product of an element of $Z({\U}(n))$ and $N$-many elements of $\U(n)$ each of which is a conjugate of either $u$ or $u^{-1}$.
	\end{lemma}
	
	\begin{proof}
		Let $\bar u,\bar v\in \PU(n)$ be the images of $u$ and $v$, and write
		$
		\|x\|_{1,n}:=\tau_n(|x|)$ 
		(for $x\in \M_{n}(\mathbb C)$).
		Set
		$
		\ell_{1,n}(\bar u):=\inf_{\lambda\in \mathbb T}\|1-\lambda u\|_{1,n}.
		$
		Since $\|1-\lambda u\|\le 2$ for every $\lambda\in\mathbb T$, the
		inequality $\|x\|_{1,n}\ge \|x\|_{2,n}^2/\|x\|$ gives
		\[
		\ell_{1,n}(\bar u)
		\ge \frac12\inf_{\lambda\in\mathbb T}\|1-\lambda u\|_{2,n}^2
		= \frac12\,\dHSn(u,Z({\U}(n)))^2
		\ge \frac{\varepsilon^2}{2}.
		\]
		
		By \cite{dowerkthom}*{Theorem~1.1 and Lemma~7.8}, there exists a universal
		constant $C>0$ such that every element of $\PU(n)$ is a product of at most
		\[
		C\,\frac{1+|\log \ell_{1,n}(\bar u)|}{\ell_{1,n}(\bar u)}
		\]
		conjugates of $\bar u$ and $\bar u^{-1}$. Consequently, with
		\[
		N:=\left\lceil C\,\frac{1+|\log(\varepsilon^2/2)|}{\varepsilon^2/2}\right\rceil,
		\]
		the element $\bar v$ is a product of at most $N$ conjugates of $\bar u$ and
		$\bar u^{-1}$ in $\PU(n)$. Lifting this factorization from $\PU(n)$ back to
		$\U(n)$ introduces only a central scalar, which is exactly the required
		statement.
	\end{proof}
	
	For an involution $a\in \U(k)$, let $m_-(a)$ denote the multiplicity of
	the eigenvalue $-1$. Then, we define the \emph{normalized trace} of $a$ by
	\[
	\ntr(a):=1-\frac{2m_-(a)}{k} \in [-1,1].
	\]
	Note that $\ntr(a)$ is just the normalized trace of $a$ as an element of the matrix algebra $\M_k(\bbC)$.
	Also note that for $u \in \UHS[(k_n)_n]$, its normalised trace $\ntr(u)$ is well defined as an element of the reduced product $[-1,1]^\bbN / \Fin$. The metric reduced power $[a,b]^\bbN / \Fin$ of a real interval $[a,b]$ has as elements the $[a,b]$-valued sequences considered up to addition of an element of $c_0$. Moreover on $[0,1]^\bbN / \Fin$ the order relation $\leq$ and the following two binary operations $\oplus$ and $\odot$ are well defined:
	\begin{align*} [x_n]_n\leq [y_n]_n&\Leftrightarrow \max(x_n-y_n,0)\to 0,\\
		[x_n]_n\oplus [y_n]_n &=[\min(x_n+y_n,1)]_n,\\
		[x_n]_n\odot [y_n]_n&=[\max(x_n+y_n-1,0)]_n.\end{align*}
	It is straightforward to check that each of the two operations $\oplus$,$\odot$ is definable in $[0,1]^\bbN / \Fin$ from the order $\leq$ and the other operation.
	The following lemma shows that certain inequalities of the absolute normalized trace can be expressed in group theoretic terms in the groups $\UHS[(k_n)_n]$.
	\begin{lemma}\phantomsection\label{flower}
		\leavevmode
		\begin{enumerate}
			\item \label{eq:a} 
			The following are equivalent for involutions $u,c\in \UHS[(k_n)_n]$ :
			\begin{enumerate}
				\item \label{eq:a1} $\exists g,h\in\UHS[(k_n)_n]$ such that $u^gu^h=c$,
				\item \label{eq:a2} $2\, |\ntr(u)| -1 \leq\ntr(c)$.
			\end{enumerate}
			\item \label{eq:b}    The following are equivalent for involutions $u,v,c \in \UHS[(k_n)_n]$: 
			\begin{enumerate}
				\item \label{eq:b1} $\exists g_1,g_2,h_1,h_2\in\UHS[(k_n)_n]$ such that $u^{g_1},u^{g_2},v^{h_1},v^{h_2}$ all commute and 
				
				$u^{g_1}u^{g_2}v^{h_1}v^{h_2}=c$,
				\item \label{eq:b2} $2(|\ntr(u)|+|\ntr(v)|)-3\leq\ntr(c)$.
			\end{enumerate}
		\end{enumerate}
	\end{lemma}
	\begin{proof}
		We first prove part \eqref{eq:a}. Choose representing sequences of exact
		involutions $u_n,c_n$. Replacing $u_n$ by $-u_n$ if necessary does not
		change $|\ntr(u_n)|$ or the product of two conjugates, so we may assume
		$\ntr(u_n)\ge 0$ for all $n$.
		
		Assume first that $2\,|\ntr(u)|-1\le \ntr(c)$. Fix $\varepsilon>0$. For all
		sufficiently large $n$ one then has
		$
		m_-(c_n)\le 2m_-(u_n)+\varepsilon k_n.
		$
		We may without loss of generality assume $c_n$ is diagonal. Choose two diagonal conjugates
		$u_n',u_n''$ of $u_n$ such that the symmetric difference of their $-1$-supports differs from the $-1$-support of $c_n$ in at most $\varepsilon k_n + 1$
		coordinates. Then $c_n':=u_n'u_n''$ is an involution and
		\[
		\dHSk{k_n}(c_n',c_n)^2\le 4(\varepsilon + \frac{1}{k_n}).
		\]
		Hence it is possible to select for every $n$ a unitary matrix $c'_n$ that is the product of two conjugates of $u_n$ in such a way that $[c_n']_n=c$, which proves \eqref{eq:a1}.
		
		Conversely, suppose $c=u^gu^h$. Choose, using Lemma~\ref{th.stababgroup}, representing sequences of exact
		involutions $a_n,b_n,c_n$ for $u^g,u^h,c$, with $c_n=a_nb_n$. Then
		\[
		\operatorname{rank}(c_n-1)
		=\operatorname{rank}(a_nb_n-1)
		\le \operatorname{rank}(a_n-1)+\operatorname{rank}(b_n-1)
		=2m_-(u_n) + o(k_n).
		\]
		Since $\operatorname{rank}(c_n-1)=m_-(c_n)$, this yields
		$
		m_-(c_n)\le 2m_-(u_n) + o(k_n),
		$
		which is exactly $2\,|\ntr(u)|-1\le \ntr(c)$.
		
		We now prove part \eqref{eq:b}. Use Lemma~\ref{th.stababgroup} to
		choose for $u,v$ representing sequences of exact involutions $u_n,v_n$ such that
		$u_nv_n=v_nu_n$ for all $n$, and choose any representing sequence of exact
		involutions $c_n$ for $c$. Replacing $u_n$ and $v_n$ by $-u_n$ and $-v_n$
		if necessary, we may again assume $\ntr(u_n),\ntr(v_n)\ge 0$. We may also assume again that each $c_n$ is diagonal.
		
		Assume first that \eqref{eq:b2} holds. Fix $\varepsilon>0$. For all
		sufficiently large $n$ we then have
		$
		m_-(c_n)\le 2m_-(u_n)+2m_-(v_n)+\varepsilon k_n.
		$
		
		Choose diagonal conjugates $u_n',u_n''$ of $u_n$ and $v_n',v_n''$ of
		$v_n$ such that, defining $D_1$ to be the symmetric difference of the $-1$-supports of $u_n',u_n''$ and $D_2$ to be the symmetric difference of the $-1$-supports of $v_n',v_n''$, we have that $D_1$ and $D_2$ are disjoint and also that $D_1 \sqcup D_2$ differs from the
		$-1$-support of $c_n$ in at most $\varepsilon k_n + 1$ coordinates. Then the
		four factors commute and
		$
		c_n':=u_n'u_n''v_n'v_n''
		$
		satisfies $\dHSk{k_n}(c_n',c_n)^2\le 4(\varepsilon+ 1 /k_n)$. Again it follows that one can select suitable unitary matrices $c'_n$ for every $n$ with
		$[c_n']_n=c$, which proves \eqref{eq:b1}.
		
		Conversely, suppose \eqref{eq:b1} holds. Let
		$u^{(i)}:=u^{g_i}$ and $v^{(i)}:=v^{h_i}$. By Lemma~\ref{th.stababgroup}
		we may choose representing sequences of exact involutions
		$u_n^{(1)},u_n^{(2)},v_n^{(1)},v_n^{(2)}$ and $c_n$ such that all four
		factors commute and
		$
		c_n=u_n^{(1)}u_n^{(2)}v_n^{(1)}v_n^{(2)}$ (for $n\in\bbN$).
		After simultaneous diagonalization, the multiplicity of $-1$ in $c_n$ is
		at most the sum of the multiplicities of $-1$ in the four factors. Hence
		$
		m_-(c_n)\le 2m_-(u_n)+2m_-(v_n) + o(k_n),
		$
		which translates exactly into \eqref{eq:b2}.		
	\end{proof}

	We prove part \eqref{itm1} of Theorem \ref{th.action boundary} by proving the following proposition.
	
	\begin{proposition}
		Let $\varphi:\UHS[(k_n)_n]\to\UHS[(l_n)_n]$ be a group isomorphism.
		\begin{enumerate}
			\item\label{eq:1} There exists $\theta\in\Aut(\partial \beta \omega)$ such that for every involution $u\in \UHS[(k_n)_n]$ and every $\cU\in\partial\beta\omega$:
			\[ \lim_{n\to \cU}|\ntr(u'_n)| =\lim_{n\to \theta(\cU)} |\ntr(u_n)|,  \]
			where $\varphi(u)=[u'_n]_n$.
			\item\label{eq:2} With $\theta$ as in \eqref{eq:1}, for all $u,v\in\UHS[(k_n)_n]$ and $S\in \pow(\bbN)/\Fin$:
			\[ u=_S v\Rightarrow \varphi(u)\varphi(v)^{-1}\in Z_{\theta(S)}(\UHS[(l_n)_n]).   \]
			\item \label{eq:3} With $\theta$ as in \eqref{eq:1}, for all $u,v\in \UHS[(k_n)_n]$ and $S\in \pow(\bbN)/\Fin$:
			\[  u=_S v \Leftrightarrow \varphi(u)=_{\theta(S)} \varphi(v). \]
		\end{enumerate}
	\end{proposition}
	
	\begin{proof}[Proof of \eqref{eq:1}]
		Set $G:=\UHS[(k_n)_n]$, $G^{(2)}=\{ u\in G: u^2=1  \}$.
		Note that the absolute normalised trace defines a map \[\antr: \{ u^G:u\in G^{(2)}  \}\to [0,1]^\bbN/\Fin: ([u_n]_n)^G \mapsto [|\ntr(u_n)|]_n.\]
		Moreover, by Lemma~\ref{flower} we have for all $u,v \in G^{(2)}$:
		\[ \antr(u^G)=\antr(v^G)\Leftrightarrow u^Gu^G \cap G^{(2)} =v^Gv^G \cap G^{(2)},  \]
		\[ \antr(u^G) \leq \antr(v^G)\Leftrightarrow u^Gu^G \cap G^{(2)} \supseteq  v^Gv^G\cap G^{(2)}, \]
		and for all $u_1,v_1,u_2,v_2\in G^{(2)}$:
		\begin{align*} &\antr(u_1^G)\odot \antr(v_1^G)=\antr(u_2^G)\odot\antr(v_2^G)\\
			\Leftrightarrow\  &\{ u_1^{g_1}u_1^{g_2}v_1^{h_1}v_1^{h_2} : u_1^{g_1},u_1^{g_2},v_1^{h_1},v_1^{h_2} \text{ all commute} \} \\ &=   \{ u_2^{g_1}u_2^{g_2}v_2^{h_1}v_2^{h_2} : u_2^{g_1},u_2^{g_2},v_2^{h_1},v_2^{h_2} \text{ all commute} \}. \end{align*}
		Likewise, we get a map $\antr:\{ u^G:u\in \UHS[(l_n)_n]\}\to [0,1]^\bbN/\Fin$ and the MV-structure $([0,1]^\bbN/\Fin,\leq,\odot, \oplus)$ is similarly definable in $\UHS[(l_n)_n]$, so there exists an automorphism $\eta$ of $([0,1]^\bbN/\Fin,\leq,\odot, \oplus)$ such that
		$ \antr(\varphi(u)^G)=\eta(\antr(u^G))$ for all $u\in G^{(2)}. $
		By \cite{cignoli}*{Proposition~4.2} there exists $\theta\in \Aut(\partial\beta\omega)$ such that for all $x\in [0,1]^\bbN/\Fin$, $\cU\in\partial\beta\omega$: $ \lim_{n\to\cU}\eta(x)=\lim_{n\to\theta(\cU)}x,  $
		and \eqref{eq:1} follows.
	\end{proof}
	
	\begin{proof}[Proof of \eqref{eq:2}]
		Suppose there were $u\in \UHS[(k_n)_n]$ and an infinite $S\subseteq \bbN$ such that
		$u=_S 1$ but $\varphi(u)\notin Z_{\theta(S)}(\UHS[(l_n)_n]).$
		Then there is $T\subseteq S$ infinite such that
		\[ \varepsilon:= \liminf_{n\in\theta(T)} \dHSk{l_n}(u'_n,Z(\U(l_n)))>0, \]
		where $u'=\varphi(u)$.
		
		Using Lemma~\ref{th.boundedgeneration}, we find
		$g'_1,\ldots,g'_N,h'_1,\ldots,h'_N\in\UHS[(l_n)_n]$
		and involutions $v',w'\in\UHS[(l_n)_n]$ such that
		$v'w'\neq w'v'$, $v'=_{\theta(T)^c}1$, and $w'=_{\theta(T)^c}1$.
		Moreover,
		\begin{equation}\label{dd}
			\begin{aligned}
				v'Z &= u'^{g'_1}u'^{g'_2}\ldots u'^{g'_N}Z,\\
				w'Z &= u'^{h'_1}u'^{h'_2}\ldots u'^{h'_N}Z.
			\end{aligned}
		\end{equation}
		with $Z = Z(\UHS[(l_n)_n])$.
		Let    $v'=\varphi(v)$, $w'=\varphi(w)$, then, using \eqref{eq:1} we find that 
		$ \lim_{n \in T^c} \dHSk{k_n}(v_n, \{I_{k_n}, -I_{k_n} \}) =  \lim_{n \in T^c} \dHSk{k_n}(w_n, \{I_{k_n}, -I_{k_n} \})= 0$     and $vw\neq wv$, but by (\ref{dd}) also $v,w\in Z_T$, a contradiction.
	\end{proof}
	
	\begin{proof}[Proof of {\eqref{eq:3}} from {\eqref{eq:2}}]
		
		For $S \in \pow(\bbN) / \Fin,$    define
		$ \rho_S: \UHS[(k_n)_n] \to \UHS[(k_n)_{n\in S}] $ by $\rho_S( [u_n]_n) = [u_n]_{n\in S}$
		and
		$ \iota_S: \UHS[(k_n)_{n\in S}] \to \UHS[(k_n)_n]  ,$
		by $\rho_S(\iota_S(u)) = u$ and $\iota_S(u) =_{S^c} 1$ for all $u \in \UHS[(k_n)_{n\in  S}]$. Both $\rho_S$ and $\iota_S$ are homomorphisms.
		Using ${\eqref{eq:2}}$, we find that $\chi_S := \rho_{\theta(S)} \circ \varphi \circ \iota_{S^c}$ defines a homomorphism $\chi_S: \UHS[(k_n)_{n\in S^c }]\to \prod_{n\in \theta(S)} Z(\U(l_n))/\Fin$. However, $\UHS[(k_n)_{n\in S^c}] \cong \prod_{n\in S^c} \SU(k_n) / \Fin$ has the property that every element is a commutator, therefore every homomorphism from $\UHS[(k_n)_{n\in S^c}] $ to an abelian group is trivial. In particular, $\chi_S$ is trivial.
		
		Then for all $u\in \UHS[(k_n)_n]$ with $u=_S 1$, we have
		$u=\iota_{S^c}(\rho_{S^c}(u))$, and hence $\rho_{\theta(S)}(\varphi(u)) =\chi_S(\rho_{S^c}(u)) =1_{\theta(S)}    ,$ therefore $\varphi(u) =_{\theta(S)} 1$.
		
	\end{proof}
	
	\section{Metric reduced products of matrix algebras}\label{sec:matrix-reduced-products}
	
	There are now essentially two routes to the study of $*$-homomorphisms between metric reduced products of matrix algebras. One is to note that such a $*$-homomorphism induces a homomorphism between the unitary groups, which can be approached using the methods from the foregoing sections and uniquely determines the $*$-homomorphism.  
	The other route is to prove coordinate recognition directly and apply Ulam stability for matrix algebras. Both routes are viable and have their own complications, see \cite{AlekseevThom}.
	
	In agreement with the earlier introduced notation for metric groups (see Notation~\ref{not.not1}), for $a=[a_n]_n,b=[b_n]_n\in \MHS[(k_n)_n]$ and
	$S\in \pow(\bbN)/\Fin$, we write
	$a=_S b$
	if
	$\lim_{n\in S}\|a_n-b_n\|_{2,k_n}=0.$
	
	Again, there is a canonical identification of asymptotically equivalent dimensions: for $(k_n)_n,(l_n)_n \in \seqtoinf$ sequences with $\frac{k_n}{l_n} \to 1$ and $(a_n)_n \in \MHS[(k_n)_n]$ we write $(a_n)_n\updownarrow \MHS[(l_n)_n]$ for the unique element $[b_n]_n$ of $\MHS[(l_n)_n]$ with  \[\lim_n
	\dHSk{k_n\vee l_n}
	\bigl(a_n\oplus 0_{k_n\vee l_n-k_n},b_n\oplus 0_{k_n\vee l_n-l_n}\bigr)=0.\]
	
	We first consider two results that can be deduced straightforwardly from the work already done for unitary groups.
	
	\begin{theorem}\label{th.rig iso alg}
		Assume $\mathrm{OCA}+\mathrm{MA}_{\aleph_1}(\sigma\text{-linked})$.
		Then every unital \hbox{$*$-isomorphism}
		\[
		\varphi: \MHS[(k_n)_n]\to \MHS[(l_n)_n]
		\]
		is trivial, i.e.\ there exists an almost permutation $f : \bbN \to \bbN$ such that  $\displaystyle\lim_{n\to\infty} l_n/ k_{f(n)}=1$ and there exists $w \in \UHS[(l_n)_n]$ such that $\varphi$ maps every $[u_n]_n \in \MHS[(k_n)_n]$ to $w^*vw \in \MHS[(l_n)_n]$, where $v =[v_n]_n = [u_{f(n)}]_n\updownarrow \MHS[(l_n)_n].$
	\end{theorem}
	\begin{proof}
		Let $\varphi: \MHS[(k_n)_n]\to \MHS[(l_n)_n]$ be a unital $*$-isomorphism. Since the unitary group of the $C^*$-algebra $\MHS[(k_n)_n]$ consists precisely of the classes $[u_n]_n$ with $u_n \in \U(k_n)$, $\varphi$ restricts to a group isomorphism
		$ \varphi_{\U} : \UHS[(k_n)_n]\to \UHS[(l_n)_n].$
		Applying Theorem~\ref{th.rigunitary} yields an almost permutation $f: \mathbb{N} \to \mathbb{N}$ such that $k_{f(n)}/l_n\to 1$ and such that there exists a sequence $(\alpha_n)_n$ with
		$
		\alpha_n\in \Aut(\U(l_n))
		$
		and for every $[u_n]_n\in \UHS[(k_n)_n]$,
		$
		\varphi_{\U}([u_n]_n)=[\alpha_n(v_n)]_n,
		$
		where $[v_n]_n = [u_{f(n)}]_n\updownarrow \UHS[ (l_n)_n]$ is the image of $[u_{f(n)}]_n$ under asymptotic-dimension identification.
		Since $\varphi$ fixes scalar multiples of the identity, we find that for all but finitely many $n$ there exists $w_n \in \U(l_n)$ such that  $\alpha_n$ is the inner automorphism of $\U(l_n)$ given by conjugation with $w_n$. Defining $w = [w_n]_n \in \UHS[(l_n)_n]$, it follows that the map
		\[  \MHS[(k_n)_n]\to \MHS[(l_n)_n] : [u_n]_n \mapsto \Ad_{w}([u_{f(n)}]_n\updownarrow \MHS[(l_n)_n])\]
		is a  $*$-isomorphism that coincides with $\varphi$ on all unitary elements of $\MHS[(k_n)_n]$, and therefore this map is actually equal to $\varphi$.	
	\end{proof}
	
	Similarly, we deduce triviality (in $\mathsf{ZFC}$) of all product form isomorphisms between tracial reduced product $C^*$-algebras $\MHS[(k_n)_n]$ and $\MHS[(l_n)_n]$: after the canonical identification of asymptotically equivalent dimensions every such isomorphism is induced coordinatewise by inner automorphisms of the matrix algebras $\M_{l_n}(\mathbb C)$.

	\begin{proposition}\label{prop.prod form alg}
		Let
		$
		\varphi:\MHS[(k_n)_n]\longrightarrow \MHS[(l_n)_n]
		$
		be a unital $*$-isomorphism of product form. Then
		$k_n/l_n \to 1$ and there exist unitaries $w_n\in \U(l_n)$ such that for all $[a_n]_n \in \MHS[(k_n)_n]$, if
		$
		[b_n]_n=[a_n]_n \updownarrow \MHS[(l_n)_n],
		$
		then
		$
		\varphi([a_n]_n)=[w_n^*b_nw_n]_n.
		$
	\end{proposition}
	
	\begin{proof}
		As in the proof of Theorem~\ref{th.rig iso alg}, we first find that $\varphi$ restricts to a group isomorphism
		$ \varphi_{\U} : \UHS[(k_n)_n]\to \UHS[(l_n)_n].$
		Furthermore, we claim that $\varphi_{\U}$ is again of product form. Let's fix a sequence $(h_n)_n$ of maps $h_n : \M_{k_n}(\mathbb{C}) \to \M_{l_n}(\mathbb{C})$ such that $\varphi = \varphi[(h_n)_n]$. We note that then, since $\varphi$ maps unitary elements of $\MHS[(k_n)_n]$ to unitary elements of $\MHS[(l_n)_n]$: $\sup_{u\in \U(k_n)} \|h_n(u)^*h_n(u)-1\|_{2,l_n} \to 0$.
		
		
		An elementary argument considering the polar decomposition of $h_n(u)$ then gives that for every $n$ we can fix a map $\tilde{h}_n: \U(k_n) \to \U(l_n)$ such that
		\[
		\sup_{u\in \U(k_n)}\|\widetilde h_n(u)-h_n(u)\|_{2,l_n} \to 0.
		\]
		Hence for every sequence $(u_n)_n$ of unitaries $u_n \in \U(k_n)$:
		$
		[h_n(u_n)]_n=[\widetilde h_n(u_n)]_n,
		$
		so the restriction $\varphi_{\U}$ is induced by the sequence of maps
		$
		\widetilde h_n:\U(k_n)\to \U(l_n).
		$
		Now Proposition~\ref{prop_stab} applies to $\varphi_{\U}$. It yields:
		\begin{enumerate}
			\item $k_n/l_n \to 1$;
			\item there exists a sequence $(\alpha_n)_n$ with
			$
			\alpha_n\in \Aut(\U(l_n))
			$
			such that for every unitary $[u_n]_n\in \UHS[(k_n)_n]$,
			$
			\varphi_{\U}([u_n]_n)=[\alpha_n(v_n)]_n,
			$
			where 
			$
			[v_n]_n  = [u_n]_n \updownarrow \UHS[(l_n)_n].
			$
		\end{enumerate}
		We conclude just as in the proof of Theorem~\ref{th.rig iso alg}: since $\varphi$ is the identity on scalar multiples of the identity, for all but finitely many $n$ the automorphism $\alpha_n$ is inner, therefore $\varphi$ coincides after the asymptotic-dimension identification with an inner automorphism of $\MHS[(k_n)_n]$ on all unitary elements. Hence $\varphi$ is itself inner up to identification of asymptotically equivalent dimensions.
	\end{proof}
	
	The analysis of product form $*$-homomorphisms between tracial reduced products of matrix algebras is very closely tied to the study of Ulam stability for $*$-homomorphisms between matrix algebras with respect to the trace norm, which was recently pursued by Vadim Alekseev and the second author in \cite{AlekseevThom}.
	
	We conclude this paper by a further exploration of the second mentioned route towards rigidity phenomena for tracial reduced product matrix algebras: proving coordinate recognition directly and then applying such Ulam stability for matrix algebras (relying on a result from \cite{AlekseevThom}). Following this route one expects to be able to further exploit the rich $C^*$-algebraic structure. We first state the relevant Ulam stability theorem, rephrased from \cite{AlekseevThom}, that we will make use of for this purpose. 
	
	\begin{definition}\label{def:matrix-approx-star}
		Let $\varepsilon>0$ and let $\varphi\colon \M_{n}(\mathbb C)\to \M_{m}(\mathbb C)$ be a map. We say that $\varphi$ is an \emph{$\varepsilon$-trace-$*$-homomorphism} if for all $a,b\in \M_n(\mathbb C)$ with $\|a\|,\|b\|\le 1$ and all $\lambda\in\bbC$ with $|\lambda|\le 1$ one has
		\begin{align*}
			\|\varphi(a+b)-\varphi(a)-\varphi(b)\|_{2,m} &\le \varepsilon,\\
			\|\varphi(\lambda a)-\lambda\varphi(a)\|_{2,m} &\le \varepsilon,\\
			\|\varphi(ab)-\varphi(a)\varphi(b)\|_{2,m} &\le \varepsilon,\\
			\|\varphi(a^*)-\varphi(a)^*\|_{2,m} &\le \varepsilon.
		\end{align*}
		If in addition
		$
		\|\varphi(1)-1\|_{2,m}\le \varepsilon,
		$
		then we call $\varphi$ \emph{$\varepsilon$-unital}.
	\end{definition}
	
	\begin{theorem}\label{alekseevthom}\cite[Theorem 3.5]{AlekseevThom}
		For every $\varepsilon>0$ there exists $\delta >0$ such that the following holds.
		
		For every $n,m\in \mathbb{N}$ and every map
		$ \varphi:\M_n(\mathbb{C})\to\M_m(\mathbb{C})  $
		which is a $\delta$-trace-$*$-homomorphism on the operator norm unit ball and is $\delta$-unital, there exists a map of the form
		\[ \psi(a)=w^* \,\mathrm{diag}(\underbrace{a,\ldots,a}_{r\text{ times}},0)w\quad (a\in \M_n(\mathbb{C})),  \]
		with $w\in\U(m')$, $r\geq 0$, and an isometry $W:\mathbb{C}^m\to \mathbb{C}^{m'}$ such that $ |m'-m|\leq \varepsilon m   $
		and
		\[  \sup_{\|a\|\leq 1}\| \psi(a)W-W\varphi(a)\|_{2,m'} \leq \varepsilon.  \]
		In particular, after stabilizing the target by codimension $o(m)$, the map $\varphi$ is approximated on the operator-norm unit-ball by a genuine $*$-homomorphism into a corner.	
	\end{theorem}
	
	We will further on write $a^{\oplus r}$ for \[\mathrm{diag}(\underbrace{a,\ldots,a}_{r\text{ times}}) \in \M_{kr} (\bbC), \quad (a\in \M_k(\bbC)). \]
	
	The proof of this Ulam stability result uses a refinement of the approach taken here to prove Proposition~\ref{prop.prod form alg}: consider the restriction of the $\delta$-trace-$*$-homomorphism to the unitary group, first correct this to a genuine continuous representation by a small perturbation which can then be further analyzed through representation theoretic means and finally lift again to the full matrix algebra. Performing this strategy on the individual factors rather than on the full reduced product allows for a finer result which in turn can be used to infer stronger statements on product form homomorphisms between tracial reduced product $C^*$-algebras. From Theorem~\ref{alekseevthom} one can deduce the following strengthening of Proposition~\ref{prop.prod form alg}.
	
	\begin{theorem}~\label{th.prod form alg non iso}
		Let $(k_n)_n$ and $(l_n)_n$ be sequences of positive integers with $k_n,l_n\to\infty$, and let $\varphi:\MHS[(k_n)_n]\to \MHS[(l_n)_n]$
		be a unital $*$-homomorphism of product form. 
		Then there exist:
		\begin{itemize}
			\item integers $m_n,r_n$ with $m_n = k_n r_n$ and
			$ m_n /l_n \to 1$;
			\item unitaries $w_n\in \U(m_n)$;
		\end{itemize}
		such that, defining
		\[
		\alpha_n(x)
		:=
		w_n^*
		x^{\oplus r_n}
		w_n
		\qquad (x\in \M_{k_n}(\mathbb C)),
		\]
		one has:
		$\varphi([x_n]_n) = [\alpha_n(x_n)]_n \updownarrow M^{HS}_\infty[(l_n)_n]$ for all $[x_n]_n \in M^{HS}_\infty[(k_n)_n]$.
	\end{theorem}
	In particular, all unital $*$-homomorphisms of product form between tracial reduced products of matrix algebras are injective.
	
	\begin{proof}
		Fix maps
		$
		h_n:M_{k_n}(\mathbb C)\to M_{l_n}(\mathbb C)
		$
		such that $\varphi = \varphi[(h_n)_n]$. Since $\varphi$ is a unital $*$-homomorphism, there exists a sequence $(\delta_n)_n$ with $
		\delta_n\searrow 0$
		such that $h_n$ is a $\delta_n$-trace-$*$-homomorphism on the operator-norm unit ball and is $\delta_n$-unital. 
		
		We apply Theorem~\ref{alekseevthom} (trace-norm Ulam stability for matrix algebras) to the maps $h_n$.  Since $\delta_n\to0$, after discarding finitely
		many coordinates we obtain:
		\begin{itemize}
			\item integers $m'_n$ with
			$
			{m'_n}/{l_n}\to1;
			$
			\item isometries
			$
			W_n:\mathbb C^{l_n}\to\mathbb C^{m'_n};
			$
			\item finite-dimensional $*$-homomorphisms
			$
			\psi_n:M_{k_n}(\mathbb C)\to M_{m'_n}(\mathbb C);
			$
		\end{itemize}
		such that
		$
		\sup_{\|x\|\le1}
		\|\psi_n(x)W_n-W_nh_n(x)\|_{2,m'_n}\to 0,
		$
		where each $\psi_n$ has the
		standard form
		$
		\psi_n(x)=w_n^*
		\bigl(x^{\oplus r_n}\oplus 0_{m'_n-k_nr_n}\bigr)w_n.
		$
		
		Put	$m_n:=k_nr_n$.
		We show next that ${m_n}/{l_n}\to1,$ using the unitality of $h_n$.  Let $P_n:=W_nW_n^*$
		be the projection onto the range of $W_n$, and let $E_n:=\psi_n(1)$.
		The estimate above applied to $x=1$ gives $\|E_nW_n-W_nh_n(1)\|_{2,m'_n}\to0.$
		Since $\|h_n(1)-1\|_{2,l_n}\to0$, we get $\|E_nP_n-P_n\|_{2,m'_n}\to0.$
		Hence $P_n$ is asymptotically contained in the range of $E_n$.  Since
		$\operatorname{rank}(P_n)=l_n,
		\operatorname{rank}(E_n)=m_n,$
		and since already $m'_n/l_n\to1$, it indeed follows that
		$
		\frac{m_n}{l_n}\to1.
		$
		
		At this point we argue through similar estimates as in Lemma~\ref{lem:product-form-stabilization}:
		if two maps agree after compression to a subspace of asymptotically full
		dimension, then they induce the same map on the reduced product.  For
		every $\|x\|\le1$, we have that $\psi_n(x)$
		is close in normalized Hilbert--Schmidt norm to	$W_nh_n(x)W_n^*$
		up to the complement of $P_n$, and that complement has relative dimension
		$
		\frac{m'_n-l_n}{m'_n}\to0.
		$
		Thus
		$
		[h_n(x_n)]_n
		=
		[\psi_n(x_n)]_n
		\updownarrow
		M^{\mathrm{HS}}_\infty[(l_n)_n]
		$
		for every sequence of contractions $(x_n)_n$.
		Moreover, since $\psi_n(x)=w_n^*(x^{\oplus r_n}\oplus 0_{m'_n-m_n})w_n$
		and
		$
		\frac{m'_n-m_n}{m'_n}\to0,
		$
		the zero summand is invisible in the reduced product and the same
		reduced-product map is induced by
		\[
		\alpha_n:M_{k_n}(\mathbb C)\to M_{m_n}(\mathbb C),
		\qquad
		\alpha_n(x)=w_n^*x^{\oplus r_n}w_n,
		\]
		after the canonical comparison between the asymptotically equivalent dimension
		sequences $(m_n)_n$ and $(l_n)_n$.
		
		Hence, for every sequence of contractions $(x_n)_n$,
		$
		\varphi([x_n]_n)
		=
		[\alpha_n(x_n)]_n
		\updownarrow
		M^{\mathrm{HS}}_\infty[(l_n)_n].
		$
		Finally, by linearity the same equality holds for every operator-norm bounded
		sequence $(x_n)_n$. 
		This proves the theorem.		
	\end{proof}

	In addition to strong stability properties, tracial reduced products of matrix algebras exhibit a very strong form of coordinate recognition simply because the center of $\MHS[(k_n)_n]$ is precisely $\ell_\infty/c_0\cong C(\partial\beta\bbN)$.
	
	\begin{lemma}\label{lem:mhs-center}
		The center of $\MHS[(k_n)_n]$ consists precisely of the classes
		$
		[\lambda_n I_{k_n}]_n$
		with $(\lambda_n)_n\in \ell_\infty$.
		Consequently,
		$
		Z(\MHS[(k_n)_n])\cong \ell_\infty/c_0\cong C(\partial\beta\bbN),
		$
		and the central idempotents are exactly the elements
		$
		e_S:=[1_S(n)I_{k_n}]_n$
		with $S\in \pow(\bbN)/\Fin$.
	\end{lemma}
	
	\begin{proof}
		Every bounded scalar sequence clearly defines a central element. Conversely, let
		$a=[a_n]_n\in Z(\MHS[(k_n)_n])$. For each $n$, averaging over $\U(k_n)$ gives
		\begin{align*}
			\|a_n-\tau_{k_n}(a_n)I_{k_n}\|_{2,k_n}
			&= \left\|\int_{{\U}(k_n)} \bigl(a_n-ua_nu^*\bigr)\,d\mu_n(u)\right\|_{2,k_n} \\
			&\le \sup_{u\in {\U}(k_n)}\|a_n-ua_nu^*\|_{2,k_n}.
		\end{align*}
		where $\mu_n$ denotes the Haar probability measure on $\U(k_n)$. Choose
		$u_n\in {\U}(k_n)$ with
		\[
		\|a_n-\tau_{k_n}(a_n)I_{k_n}\|_{2,k_n}
		\le \|a_n-u_na_nu_n^*\|_{2,k_n}+1/n.
		\]
		Set $u:=[u_n]_n$. Then $u\in \UHS[(k_n)_n]\subseteq\MHS[(k_n)_n]$. Since
		$a$ is central, we have $ua=au$, hence
		$
		\lim_n \|a_n-u_na_nu_n^*\|_{2,k_n}=0.
		$
		Therefore
		$
		\lim_n \|a_n-\tau_{k_n}(a_n)I_{k_n}\|_{2,k_n}=0,
		$
		so $a=[\tau_{k_n}(a_n)I_{k_n}]_n$ is represented by a scalar sequence.
		
		For scalar sequences the normalized $2$-norm agrees with the absolute value,
		so the center is exactly $\ell_\infty/c_0$. The identification with
		$C(\partial\beta\bbN)$ is the classical Gelfand--Stone identification. Under
		this identification, idempotents correspond to characteristic functions of
		clopen subsets of $\partial\beta\bbN$, equivalently to elements of
		$\pow(\bbN)/\Fin$. This gives the claimed description of the~$e_S$.
	\end{proof}

	\begin{theorem}\label{th:mhs-boundary}
		Let $(k_n)_n$ and $(l_n)_n$ be sequences of natural numbers with
		$\lim_n k_n=\lim_n l_n=\infty$. Let
		$
		\varphi\colon \MHS[(k_n)_n]\to \MHS[(l_n)_n]
		$
		be a unital ring isomorphism. Then there exists a unique automorphism
		$\theta\in \Aut(\partial\beta\bbN)$ such that for all
		$a,b\in \MHS[(k_n)_n]$ and all $S\in \pow(\bbN)/\Fin$,
		\[
		a=_S b \Longleftrightarrow \varphi(a)=_{\theta(S)}\varphi(b).
		\]
		Consequently, for every nonprincipal ultrafilter $\cU\in\partial\beta\bbN$,
		$\varphi$ induces a ring isomorphism between the corresponding metric
		ultraproducts along $\cU$ and $\theta(\cU)$.
	\end{theorem}
	
	\begin{proof}
		For $S\in \pow(\bbN)/\Fin$, let
		\[
		e_S:=[1_S(n)I_{k_n}]_n\in Z(\MHS[(k_n)_n]),\quad e'_S:=[1_S(n)I_{l_n}]_n\in Z(\MHS[(l_n)_n]).
		\]
		By Lemma~\ref{lem:mhs-center}, every central idempotent of
		$\MHS[(k_n)_n]$ is of the form $e_T$ for a unique
		$T\in \pow(\bbN)/\Fin$. Since $\varphi$ is a unital ring isomorphism, it
		preserves the center and idempotents. Hence there is a unique Boolean
		algebra automorphism
		$
		\theta\colon \pow(\bbN)/\Fin\to \pow(\bbN)/\Fin
		$
		such that
		$
		\varphi(e_S)=e'_{\theta(S)}$
		for all $S\in \pow(\bbN)/\Fin$.
		
		By Stone duality, this Boolean algebra automorphism is equivalently an
		automorphism of $\partial\beta\bbN$, 
		given by
		$
		\theta(\cU):=\{S \in \pow(\bbN)/\Fin: \theta(S)\in \cU\}
		$
		which we denote by the same symbol.
		
		Now let $a=[a_n]_n$ and $b=[b_n]_n$ belong to $\MHS[(k_n)_n]$. Then $a=_S b\Longleftrightarrow e_S(a-b)=0.$
		Applying $\varphi$, we obtain
		\[
		e_S(a-b)=0
		\Longleftrightarrow \varphi(e_S)\bigl(\varphi(a)-\varphi(b)\bigr)=0
		\Longleftrightarrow e'_{\theta(S)}\bigl(\varphi(a)-\varphi(b)\bigr)=0,
		\]
		which is exactly
		$
		a=_S b \Longleftrightarrow \varphi(a)=_{\theta(S)}\varphi(b).
		$
		This proves the main assertion.
		
		For the ultraproduct statement,  we check as follows that for $a=[a_n]_n$ and $b=[b_n]_n$ belonging to $\MHS[(k_n)_n]$, if $[a]_{\theta(\cU)}=[b]_{\theta(\cU)}$, then $[\varphi(a)]_{\cU}=[\varphi(b)]_{\cU}$.
		If $[\varphi(a)]_{\cU} \neq [\varphi(b)]_{\cU}$, then there were $S \in \cP(\bbN) / \Fin$, $\varepsilon>0$ and representing sequences $(a'_n)_n$ and $(b'_n)_n$ for $\varphi(a)$ and $\varphi(b)$ such that $\theta(S) \in \mathcal{U}$ and $\forall n\in \theta(S)\  \|a'_n-b'_n\|_{2,k_n}> \varepsilon$. Then also $[a]_{\theta(\cU)} \neq [b]_{\theta(\cU)}$, for else, since $S \in \theta(\mathcal{U})$, there were $T\subseteq S$ infinite with $a =_T b$ and then by the equivalence proved above,
		$\varphi(a)=_{\theta(T)}\varphi(b)$ which would contradict $\theta(T) \subseteq \theta(S)$. 
		Therefore $\varphi$
		descends to a ring homomorphism on the ultraproduct quotients, and the same
		argument applied to $\varphi^{-1}$ shows that this induced map is an
		isomorphism.
	\end{proof}
	
	\begin{remark}
		We record the simple observation that part of the foregoing argument can be salvaged for an arbitrary unital ring morphism $\varphi\colon \MHS[(k_n)_n]\to \MHS[(l_n)_n]$, as long as central elements are mapped to central elements. Namely, it is still the case that a Boolean algebra morphism 
		$
		\theta\colon \pow(\bbN)/\Fin\to \pow(\bbN)/\Fin
		$
		is induced such that $\varphi(e_S)=e'_{\theta(S)}$ for all $S\in \pow(\bbN)/\Fin$.
	\end{remark}

	The next theorem says that every center-preserving unital  $*$-homomorphism $\varphi:\MHS[(k_n)_n]\to \MHS[(l_n)_n]$ between tracial reduced products can be decomposed as
	$
	\varphi=(1-p)\varphi\oplus p \varphi
	$
	for some central projection $e'_T = p \in \MHS[(l_n)_n]$ with $(1-p)\varphi$ having nonmeagre central kernel and $p\varphi$ induced coordinatewise, along a finite-to-one map $T \to \bbN$, by finite-dimensional unital $*$-homomorphisms. This theorem can be seen as a trace-norm variant of known rigidity results for reduced products of matrix algebras in the operator norm (see \cite{VignatiYilmaz}).     
	
	\begin{theorem}
		\label{thm:center-preserving-tracial-wep}
		Assume $\mathsf{OCA}+\mathsf{MA}_{\aleph_1}(\sigma\text{-linked})$.
		Let
		$
		\varphi:\MHS[(k_n)_n]\to \MHS[(l_n)_n]
		$
		be a unital $*$-homomorphism satisfying
		\[
		\varphi\bigl(Z(\MHS[(k_n)_n])\bigr)
		\subseteq
		Z(\MHS[(l_n)_n]).
		\]
		Then there exists a set $T\subseteq \bbN$ with corresponding central projection
		$
		p=e'_T\in Z(\MHS[(l_n)_n])
		$
		such that:
		\begin{enumerate}
			\item the $(1-p)$-summand of $\varphi$ has large central kernel, i.e.: the ideal
			$
			\mathcal I_{(1-p)\varphi}
			:=
			\{S\in\pow(\mathbb N):(1-p)\varphi(e_S)=0\}
			$
			is nonmeagre;
			\item the $p$-summand  
			$p\varphi : \MHS[(k_n)_n]\to p\MHS[(l_n)_n] \cong  \MHS[(l_n)_{n \in T}]$
			is trivial in the following sense. There exist a finite-to-one map $g:T\to\mathbb N$, integers $r_n,m_n$ with $m_n = r_nk_{g(n)}$ and $ m_n / l_n \to 1$ for $n \in T$, and for every $n \in T$ a $*$-homomorphism
			$
			\alpha_n:\M_{k_{g(n)}}(\mathbb C)\to \M_{m_n}(\mathbb C)
			$
			of the form
			$
			\alpha_n(x)
			=
			w_n^* x^{\oplus r_n} w_n;
			$
			with $w_n \in U(m_n)$ such that, after the canonical identification of the asymptotically equivalent dimension sequences $(m_n)_{n\in T}$ and $(l_n)_{n\in T}$,
			$
			p\varphi([a_n]_n)
			=
			[\alpha_n(a_{g(n)})]_{n\in T}
			$
			for all $[a_n]_n \in \MHS[(k_n)_n]$.
		\end{enumerate}
	\end{theorem}
	\begin{proof}
		Since
		$
		\varphi(Z(\MHS[(k_n)_n]))\subseteq Z(\MHS[(l_n)_n]),
		$
		there exists an endomorphism $\theta$ of the Boolean algebra $\pow(\omega)/\Fin$    such that
		$ \varphi(e_S) = e'_{\theta(S)}$ for all $S \in \pow(\omega)/\Fin.$	
		By \cite[Theorem 4.7]{triviso} there exist a set $T \subseteq \mathbb{N}$
		and a finite-to-one map $g: T \to \mathbb{N}$ such that:
		\begin{enumerate}
			\item $\{ S\in \pow(\bbN) : \theta(S \pm \Fin) \subseteq (T\pm \Fin)\}$ is nonmeagre,
			\item $\theta(S) \cap T =  g^{-1}[S]$ for all $S\in \pow(\bbN)/\Fin$.
		\end{enumerate}
		Considering the central projection corresponding to $T$:
		$
		p=e'_T\in Z(\MHS[(l_n)_n]),
		$
		this means that $\mathcal I_{{(1-p)\varphi}}=\{S\in\pow(\mathbb N):(1-p)\varphi(e_S)=0\}$
		is nonmeagre and that $\varphi(e_S) =_T e'_{g^{-1}[S]}$ for every $S \subseteq \bbN$.
		
		It remains to prove the structural description of the $p$-summand
		\[p\varphi : \MHS[(k_n)_n]\to p\MHS[(l_n)_n] \cong  \MHS[(l_n)_{n \in T}].\]
		We first regroup the target coordinates according to the fibres of $g$.  For
		$n\in\mathbb N$, set $F_n:=g^{-1}(\{n\})\subseteq T$, a finite subset of $\bbN$. Define
		$
		C_n:=\prod_{i\in F_n}M_{l_i}(\mathbb C),
		$
		with the provision that in case $F_n = \emptyset$, we set $C_n =\{0\}$.
		Composing the morphism $p\varphi$ with the canonical re-indexing isomorphism $\MHS[(l_n)_{n \in T}] \xrightarrow{\cong} \prod_n \prod_{i \in g^{-1}[\{n\}]} \M_{l_i}(\bbC) / \Fin$, one obtains a \hbox{$*$-homomorphism}
		$ \psi :  \MHS[(k_n)_n] \to \prod_n C_n / \Fin$ which is coordinate fixing:
		$a=_Sb \Rightarrow \psi(a) =_S \psi(b)$, for all $a,b \in \MHS[(k_n)_n], S \in \pow(\bbN).$
		
		Since $p\psi$ is a
		unital $*$-homomorphism, it is contractive in operator norm, so restricts to a map from the normalised Hilbert-Schmidt reduced product of the closed balls of radius 2 of the algebras $\M_{k_n}(\mathbb C)$ to the normalised Hilbert-Schmidt reduced product of the closed balls of radius 2  of the algebras $C_n$.  
		Applying the metric lifting theorem \cite[Theorem 2.3]{debondtvignati} under $\mathsf{OCA}+\mathsf{MA}_{\aleph_1}(\sigma\text{-linked})$ to this restriction of $\psi$, it follows that there exist maps
		$H_n: \{x\in M_{k_n}(\mathbb C):\|x\|\le 2\} \to \{y\in C_n:\|y\|\le 2\}$
		such that for every sequence $(a_n)_n$ with $a_n \in \M_{k_n}(\mathbb C), \|a_n\|\le2$ for all $n$:
		$
		\psi([a_n]_n)=[H_n(a_n)]_n.
		$
		
		Writing
		$
		H_n(x)=\bigl(h_i(x)\bigr)_{i\in F_n},
		$
		we obtain, for every $n \in \bbN$ and $i\in F_n$, a map
		$
		h_i:\{x\in M_{k_n}(\mathbb C):\|x\|\le2\}
		\to M_{l_i}(\mathbb C)
		$
		for which the following lifting identity holds:
		$
		p\varphi([a_n]_n)
		=
		[h_i(a_{g(i)})]_{i\in T}
		$
		for every sequence $(a_n)_n$ satisfying $a_n \in \M_{k_n}(\mathbb C),\|a_n\|\le2$ for all $n$. Since $p\varphi$ is a $*$-homomorphism, there exists a sequence $(\varepsilon_i)_{i\in T}$ with $\varepsilon_i\searrow 0$ such that each map $h_i$ is a $\varepsilon_i$-unital $\varepsilon_i$-trace-$*$-homomorphism. By almost linearity it then follows that each of the maps $h_i$ can be extended to the entire algebra $M_{k_{g(i)}}(\mathbb C)$ in such a way that $p\varphi([a_n]_n)= [h_i(a_{g(i)})]_{i\in T}$ for all $[a_n]_n \in \MHS[(k_n)_n]$. Since $\lim_i k_{g(i)} = \infty$, we may conclude by applying Theorem~\ref{th.prod form alg non iso} to the product-form unital $*$-homomorphism
		represented by the maps
		$
		h_i:M_{k_{g(i)}}(\mathbb C)\to M_{l_i}(\mathbb C),$ for $i\in T.
		$\end{proof}


	\section*{Acknowledgements}
		Ben De Bondt is funded by the Deutsche Forschungsgemeinschaft (DFG, German Research Foundation) under Germany's Excellence Strategy EXC 2044/2\,-390685587, Mathematics Münster: Dynamics--Geometry--Structure. 
	
	Andreas Thom acknowledges funding by the Deutsche Forschungsgemeinschaft (SPP 2026 ``Geometry at infinity'').
	
	ChatGPT 5.5 was used to assist in drafting parts of this manuscript. All content was reviewed and substantially revised by the authors, who are responsible for the final text.

\end{document}